\documentclass[10pt,twoside,english,reqno,a4paper]{amsart}

\usepackage{geometry}
\usepackage{amsmath}
\newlength{\defbaselineskip}
\usepackage{listings,graphicx,varioref,color,bm,stmaryrd}

\usepackage{amsmath}

\usepackage{mathpazo} 
\usepackage[scaled=.95]{helvet} 
\usepackage{courier}
\usepackage{amssymb}
\usepackage{mathrsfs}  
\usepackage{amscd}
\usepackage{amsfonts}
\usepackage[T1]{fontenc}
\usepackage{graphics}
\usepackage[english]{babel}
\usepackage[utf8]{inputenc} 
\usepackage[makeroom]{cancel}
\usepackage{pifont}
\usepackage{marvosym}

\usepackage{amsfonts}
\usepackage{datetime}
\usepackage{hyperref}
\hypersetup{colorlinks=true,linkcolor=red!70!black,citecolor=orange,urlcolor=blue}

\definecolor{caribbeangreen}{rgb}{0.0, 0.8, 0.6}
\definecolor{darkpastelgreen}{rgb}{0.01, 0.75, 0.24}
\definecolor{green(pigment)}{rgb}{0.0, 0.65, 0.31}
\usepackage[dvipsnames]{xcolor}
\usepackage{hyperref} 
\hypersetup{
colorlinks=true,
linkcolor= red!70!black,
citecolor= orange,
urlcolor=range, 
}

\addto\extrasenglish{%
}

\addto\extrasenglish{%
}

\addto\extrasenglish{%
}

\usepackage[capitalize,nameinlink,noabbrev,nosort]{cleveref} 

\crefname{section}{ Section}{ Sections}
\crefname{subsection}{ Subsection}{ Subsections}
\crefname{appendix}{ Appendix}{ Appendices}

\crefname{figure}{ Figure}{ Figures}

\crefname{equation}{}{}

\crefname{definition}{ Definition}{ Definitions}
\crefname{defin}{ Definition}{ Definitions}
\crefname{theorem}{ Theorem}{ Theorems}
\crefname{proposition}{ Proposition}{ Propositions}
\crefname{corollary}{ Corollary}{ Corollaries}
\crefname{remark}{ Remark}{ Remarks}
\crefname{lemma}{ Lemma}{ Lemmata}
\crefname{theoa}{ Theorem}{ Theorems}
\crefname{lemb}{ Lemma}{ Lemmata}

\usepackage{enumerate}
\usepackage{graphics}
\usepackage{pgfplots}
\pgfplotsset{width=7cm,compat=newest} 
\usepackage{graphics}
\usepackage{tikz}
\usetikzlibrary{shapes.geometric}
\usepackage{tikz-cd}
\usepackage{caption}
\usepackage{subcaption}

\usepackage{float}

\newtheoremstyle{mytheoremstyle} 
{.5em}                    
{0cm}                    
{\slshape}                   
{}                           
{\bf}             
{.}                          
{.5em}                       
{}  

\theoremstyle{mytheoremstyle}
\newtheorem{theorem}{Theorem}[section]
\newtheorem{proposition}[theorem]{Proposition}

\newtheorem{defin}[theorem]{Definition}
\newtheorem{remark}[theorem]{Remark}

\numberwithin{equation}{section}
\long\def\salta#1{\relax}

\usepackage{kantlipsum}
\allowdisplaybreaks

\usepackage[normalem]{ulem}               
\newcommand\redout{\bgroup\markoverwith
{\textcolor{bor}{\rule[0.5ex]{2pt}{0.8pt}}}\ULon}

\makeatletter

\def\og{\leavevmode\raise.3ex\hbox{$\scriptscriptstyle\langle\!\langle$~}}
\def\fg{\leavevmode\raise.3ex\hbox{~$\!\scriptscriptstyle\,\rangle\!\rangle$}}

\usepackage{mdframed}
\newmdenv[
  topline=false,
  bottomline=false,
  rightline=false,
  skipabove=\topsep,
  skipbelow=\topsep,
  leftmargin=-10pt,
  rightmargin=-10pt,
  innertopmargin=4pt,
  innerbottommargin=4pt,
  linecolor=red,
  linewidth=2pt,
]{siderules}

\def\mA{\mathscr{A}}
\def\mW{\mathscr{W}}
\def\mH{\mathscr{H}}
\def\mL{\mathscr{L}}

\usepackage{scalerel,stackengine}
\stackMath
\newcommand\reallywidehat[1]{%
\savestack{\tmpbox}{\stretchto{%
  \scaleto{%
    \scalerel*[\widthof{\ensuremath{#1}}]{\kern-.6pt\bigwedge\kern-.6pt}%
    {\rule[-\textheight/2]{1ex}{\textheight}}
  }{\textheight}%
}{0.5ex}}%
\stackon[2pt]{#1}{\tmpbox}%
}
\def\TT{{\mathbb{T}^d}}
\def\ZZ{{\mathbb{Z}^d}}
\def\pat{\partial_t}

\def\fv{\widehat{v}}

\def\fa{\widehat{\mathcal{N}}}
\def\fn{\widehat{\mathcal{N}}}

\def\Re{\mathit{Re}}

\newcommand{\Vast}{\bBigg@{3}}
\newcommand{\vast}{\bBigg@{2}}

\def\de{\delta}
\def\eps{\varepsilon}

\def\D{\Delta }
\def\vp{\varphi}

\def\al{\alpha}
\def\de{\delta}
\def\vp{\varphi}

\newcommand{\N}{\nabla}

\newcommand{\ds}{\displaystyle}

\def\qq{\qquad}
\def\q{\quad}

\definecolor{bor}{cmyk}{0.21,0.93,0.86,0.12}
\definecolor{air}{rgb}{0.178, 0.51, 0.51}
\definecolor{range}{cmyk}{0,0.599,1,0.188}

\def\ds{\displaystyle}
\definecolor{apple}{rgb}{0.55, 0.71, 0.0}

\DeclareMathOperator*{\esssup}{ess\,sup}

\usepackage{xcolor}

\title{Global existence and decay to equilibrium for some crystal surface models}

\author[R. Granero-Belinch\'{o}n]{Rafael Granero-Belinch\'{o}n}
\email{rafael.granero@unican.es}
\address{Departamento  de  Matem\'aticas,  Estad\'istica  y  Computaci\'on,  Universidad  de Cantabria.  Avda.  Los  Castros  s/n,  Santander,  Spain.}

\author[M. Magliocca]{Martina Magliocca}
\email{magliocc@mat.uniroma2.it}
\address{Dipartimento di Matematica, Universit\`a degli Studi Tor Vergata, Via della Ricerca Scientifica 1, 00133 Rome, Italy.}


\begin{document}

\begin{abstract}
In this paper we study the large time behavior of the solutions to the following nonlinear fourth-order equations
$$
\pat u=\D e^{-\D u},
$$
$$
\pat u=-u^2\D^2(u^3).
$$
These two PDE were proposed as models of the evolution of crystal surfaces by J. Krug, H.T. Dobbs, and S. Majaniemi (\emph{Z. Phys. B}, 97, 281-291, 1995) and H. Al Hajj Shehadeh, R. V. Kohn, and J. Weare (\emph{Phys. D}, 240, 1771-1784, 2011), respectively. In particular, we find explicitly computable conditions on the size of the initial data (measured in terms of the norm in a critical space) guaranteeing the global existence and exponential decay to equilibrium in the Wiener algebra and in Sobolev spaces. 
\end{abstract}

\maketitle
{\small
\tableofcontents}

\section{Introduction}
Crystal films are important in many modern electronic devices (as mobile phone antennae) and nano\-technology. Thus, the evolution of a crystal surface is an interesting topic with important applications. The purpose of this paper is to find explicitly computable conditions guaranteeing the global existence and exponential decay to equilibrium of the solutions of certain PDE models of crystal surfaces. 

\medskip

The first problem we deal with is 
\begin{equation}\label{1u} 
\begin{cases}
\begin{array}{lll}
\ds \pat u=\D e^{-\D u}& \text{ in }  [0,T]\times\TT ,&\\
\ds u(0,x)=u_0(x) &\text{ in }\TT,&
\end{array}
\end{cases}
\end{equation}
where the initial datum $u_0$ satisfies
$$
\int_\TT u_0(x)\,dx=0
$$
and $\TT=[-\pi,\pi]^d$ is the $d$-dimensional torus ($d=1,2$) and $0<T<\infty$. This model was suggested by Krug, Dobbs, \& Majaniemi \cite{krug1995adatom} (equations (4.4) and (4.5) in \cite{krug1995adatom}) (see also Marzuola \& Weare \cite{marzuola2013relaxation}) as a description (under certain physical assumptions and simplifications) of the large scale evolution of the crystal surface.

This equation has been previously studied in the mathematical literature. In particular, Liu \& Xu \cite{liu2016existence} obtained the existence of weak solutions starting from arbitrary initial data. Moreover, they define a stationary solution having a curvature singularity evidencing that the solution may develop singularities. Gao, Liu, \& Lu \cite{gao2017gradient} used the gradient flow approach to study the solution of \eqref{1u}. In their paper, the solution (which exists globally in time and emanates from an initial data with arbitrary size) is allowed to have a singularity at the level of second derivatives (the laplacian of the solution is allowed to be a Radon measure). Let us also mention that when the exponential in \eqref{1u} is linearized and the Laplacian is replaced by the $p-$Laplacian, the resulting equation has been studied by Giga \& Kohn \cite{giga2010scale} (see also the recent preprint by Xu \cite{xu2017existence}).

Short after this paper was posted, two new papers studying \eqref{1u} appeared, one by Jian-Guo Liu \& Robert Strain \cite{LS} and another by David Ambrose \cite{Ambrose}. We refer to the discussion section below for more details about these works and some words comparing our results and theirs.

\medskip

The second problem we deal with is
\begin{equation}\label{2u} 
\begin{cases}
\begin{array}{lll}
\ds\pat u=-u^2\D^2(u^3)&  \text{ in }  [0,T]\times\TT,  &\\
\ds u(0,x)=u_0(x)&\text{ in }  \TT,&
\end{array}
\end{cases}
\end{equation}
with periodic boundary conditions and where $u_0> 0$. Using that 
$$
\int_\TT \frac{1}{u(x,t)}\,dx=\int_\TT \frac{1}{u_0(x)}\,dx,
$$
we can rescale variables and, without lossing generality, assume that $\int_\TT 1/u(x,t)\,dx=1$. This equations was proposed by Shehadeh, Kohn \& Weare \cite{shehadeh2011evolution} (equation (1.1) in \cite{shehadeh2011evolution}) as a continuous description of the slope of the crystal surface (as a function of height and time) in the so-called Attachment-Detachment-Limited regime (see Appendix A in \cite{shehadeh2011evolution} and the references therein for further explanation on the physics behind the model).

When the extra term $-\alpha u^2\D^2u$ is added \eqref{2u} reads 
\begin{equation}\label{2ub} 
\begin{cases}
\begin{array}{lll}
\ds\pat u=-u^2\D^2(u^3+\alpha u)&  \text{ in }  [0,T]\times\TT,  &\\
\ds u(0,x)=u_0(x)&\text{ in }  \TT.&
\end{array}
\end{cases}
\end{equation}
This problem has been studied in the mathematical literature by Gao, Ji, Liu \& Witelski \cite{gao2017vicinal}. In particular, when the domain is one-dimensional and $\alpha>0$, by using the two Lyapunov functionals
$$
E_1(t)=\int \frac{1}{2}u^2+\alpha \log(u)\,dx,
$$
$$
E_2(t)=\int (\partial_x^2(u^3+\alpha u))^2\,dx,
$$
Gao, Ji, Liu \& Witelski proved the global existence of positive solutions to \eqref{2ub}. Furthermore, these authors also proved that $u(t)$ tends (at an unspecified rate) towards to certain constant steady state (which depends on the initial data).

When $\alpha=0$, Gao, Liu, \& Lu \cite{gao2017weak} proved the global existence of weak solutions to \eqref{2u} for non-negative initial data such that $u_0^3\in \mathscr{H}^2$. Moreover, in \cite{gao2017weak} the convergence of the weak solution towards a constant steady state is also obtained (at an unspecified rate).

We observe that \eqref{2u} is equivalent to
\begin{equation}\label{4} 
\begin{cases}
\begin{array}{lll}
\ds \pat w=\D((\D w)^{-3} )&  \text{ in }  [0,T]\times\TT , &\\
\ds w(0,x)=w_0(x) &\text{ in }  \TT,&
\end{array}
\end{cases}
\end{equation}
under the change of variables $\D w=1/u$. Equation \eqref{4} was studied by Liu \& Xu \cite{liu2017analytical}. In this paper, the authors proved the existence of global weak solutions. 

\medskip

Under the change of variable $v=\D u$, \eqref{1u} becomes
\begin{equation}\label{1v} 
\begin{cases}
\begin{array}{lll}
\ds \pat v=\D^2 e^{-v}& \text{ in }  [0,T]\times\TT ,&\\
\ds v(0,x)=v_0(x) &\text{ in }\TT.&
\end{array}
\end{cases}
\end{equation}
Due to the definition of $v$ and the periodic boundary conditions, we have that
$$
\int_\TT v_0(x)\,dx=\int_\TT \D u_0(x)\,dx=0.
$$

The interest in problem \eqref{2u} is strictly linked to 
\begin{equation}\label{2v} 
\begin{cases}
\begin{array}{lll}
\ds\pat v=\,\D^2 \left(\frac{1}{(1+v)^3} \right)  & \text{ in }  [0,T]\times\TT, &\\
\ds v(0,x)=v_0(x) & \text{ in }  \TT.&
\end{array}
\end{cases}
\end{equation}
Indeed, \eqref{2u} becomes \eqref{2v} under the change of variable $\frac{1}{u}=1+v$, with $v$ such that
$$
\int_\TT v(t,x)\,dx=0.
$$

Using the formulations \eqref{1v} and \eqref{2v} allow us to find two Lyapunov functionals $\mathcal{L}_1$ (for problem \eqref{1v}) and $\mathcal{L}_2$ (for problem \eqref{2v}):
$$
\mathcal{L}_1(v)=\int_\TT e^{-v(t,x)}\,dx,
$$
$$
\mathcal{L}_2(v)=\int_\TT \frac{1}{(1+v(t,x))^2}\,dx.
$$
Let us remark that \eqref{1v} and \eqref{2v} (when the domain is $\mathbb{R}^d$ to simplify the exposition) are invariant by the scaling
\begin{equation}\label{scaling} 
v_\lambda(t,x)=v(\lambda^4 t,\lambda x).
\end{equation}

In the rest of the paper we will use the formulations \eqref{1v} and \eqref{2v} to prove several global existence results. The main contribution is a global existence and decay in the Wiener algebra $\mathscr{A}^0$ (see \eqref{Wienerhomo} below for the proper definition), for initial data satisfying certain explicit size restrictions. We note that the scaling \eqref{scaling} leaves the Wiener algebra's norm invariant. Indeed, in the case where the domain is $\mathbb{R}^d$,
\begin{align*}
\|v_\lambda(t)\|_{\mathscr{A}^0}&=\int_{\mathbb{R}^d}|\widehat{v_\lambda}(t,k)|dk\\
&=\frac{1}{(2\pi)^{d/2}}\int_{\mathbb{R}^d}\left|\int_{\mathbb{R}^d}v(\lambda^4 t,\lambda x)e^{i\lambda x\cdot \frac{k}{\lambda}}\,dx\right|dk\\
&=\frac{1}{\lambda^d}\int_{\mathbb{R}^d}\left|\widehat{v}\left(\lambda^4 t,\frac{k}{\lambda}\right)\right|dk\\
&=\|v(\lambda^4 t)\|_{\mathscr{A}^0}.
\end{align*}
Although there are several available definitions of \emph{critical spaces} for PDEs, if the underlying PDE has a scaling invariance, \emph{critical spaces} are often introduced in the literature as the scaling invariant spaces. Here, we follow that definition, \emph{i.e.} by \emph{critical space} we mean that the scaling of the equation matches the scaling of the norm in that space. Thus, in this sense, our Theorems \ref{theorem1} and \ref{teo2ex} are global existence results in a \emph{critical space} for problems \eqref{1v} and \eqref{2v}. Besides this criticality of the space, the use of the Wiener algebra allow us to get estimates for higher space derivatives at the cost of time integrability. For instance, for the linear problem
$$
\pat f=-\Delta^2 f,
$$
using standard $\mL^2$ estimates one get the boundedness of the solution in $\mL^2(0,T;\mH^2)$. However, if one uses estimates in the Wiener algebra $\mA^0$ (which of course will impose better regularity assumptions on the initial data) one can prove boundedness in $\mL^1(0,T;\mA^4)$. Using the decay in the Wiener algebra, we also provide a global existence and decay for initial data in Sobolev spaces. The results for the original problems \eqref{1u} and \eqref{2u} follow easily (see the discussion section below). 

Our approach is very adaptable and can lead to advances in other systems of PDE. For instance, it has been used Bruell \& Granero-Belinch\'on to study the evolution of thin films in Darcy and Stokes flows \cite{BruellGranero2018} by C\'ordoba and Gancedo \cite{cordoba2007contour}, Constantin, C{\'o}rdoba, Gancedo, Rodriguez-Piazza, \& Strain \cite{constantin2016muskat} for the Muskat problem (see also \cite{gancedo2017muskat} and \cite{patel2017large}), by Burczak \& Granero-Belinch\'on \cite{burczak2016generalized} to analyze the Keller-Segel system of PDE with diffusion given by a nonlocal operator and by Bae, Granero-Belinch\'on \& Lazar \cite{bae2016global} to prove several global existence results (with infinite $L^p$ energy) for nonlocal transport equations.

\subsection{Notation \& Basic Tools }

Recalling the expression of the $k-$th Fourier coefficient of a $2\pi-$periodic function $u$
$$
\widehat{u}(k)=\frac{1}{(2\pi)^d}\int_{\TT}u(x)e^{-ix\cdot k}\,dx,
$$
we have the Fourier series representation
$$
u(t,x)=\sum_{k\in\ZZ}\widehat{u}(k,t)e^{ix\cdot k}.
$$
Note that, dropping $t$ from the notation, we have the following well-known facts:
\begin{align*}
\|u\|_{\mL^2(\TT)}^2=(2\pi)^d\sum_{k\in\ZZ}|\widehat{u}(k)|^2,\qquad
\widehat{uv}(k)=\sum_{j\in\ZZ}\widehat{u}(j)\widehat{v}(k-j).
\end{align*}
Let $n\in \ZZ^+$ and denote by
$$
\mW^{n,p}(\TT)=\left\{u\in \mL^p(\TT),\, \partial_x^{n} u\in \mL^p(\TT)\right\}
$$
the standard $\mL^p$-based Sobolev space with norm
$$
\|u\|_{\mW^{n,p}(\TT)}^p=\|u\|_{\mL^p}^p+\|\partial_x^n u\|_{\mL^p}^p,
$$
being $\partial$ a differential operator of order 1 with respect to a spatial variable. Then, we define the $\mL^2$-based Sobolev spaces $\mH^\alpha(\TT)$
\begin{equation*}\label{Sobhomo}
\mH^\alpha(\TT)=\left\{u\in \mL^2(\TT),\q \|u\|_{\mH^\al(\TT)}^2:=\sum_{k\in\ZZ}|k|^{2\alpha}|\widehat{u}(k)|^2<\infty
\right\}.
\end{equation*}
In a similar way, we consider the Wiener spaces $\mA^\alpha(\TT)$ as
\begin{equation}\label{Wienerhomo}
\mA^\alpha(\TT)=\left\{u\in \mL^1(\TT),\q \|u\|_{\mA^\al(\TT)}:=\sum_{k\in\ZZ} |k|^\alpha|\widehat{u}(k)|<\infty\right\}.
\end{equation}
We note that
$$
|\partial^n f(x)|\leq \sum_{n\in\mathbb{Z}^d}|k|^n|\hat{f}(n)|=\|f\|_{\mathscr{A}^n}.
$$
We simplify the notation rewriting the Lebesgue, Sobolev and Wiener norms as
\begin{align*}
\|u\|_{\mL^\al(\TT)}=\|u\|_{\mL^\al},& \qquad
\|u\|_{\mW^{n,p}(\TT)}=\|u\|_{\mW^{n,p}}, \\
\|u\|_{\mH^\al(\TT)}=\|u\|_{\mH^\al}, &\qq
\|u\|_{\mA^\al(\TT)}=|u|_{\al}.
\end{align*}
Similarly we set
$$
\|u\|_{\mL^\al(0,T;\mL^\beta(\TT))}=\|u\|_{\mL^\al(\mL^\beta)}.
$$
We recall the following inequalities: 
\begin{align}
|u|_{s}&\leq |u|_{0}^{1-\theta}|u|_{r}^{\theta}\q\forall 0\leq s\leq r\,,\;\theta=\frac{s}{r}.\label{interpolation}
\end{align}
We also introduce the space of Radon measures from an interval $[0,T]$ to a Banach space $X$,
$$
\mathcal{M}(0,T;X).
$$
We denote with $c,\,C$ positive constants which may vary from line to line during the proofs. Finally, for $j=1,\ldots,d$, we write 
$$
u,_{j}=\frac{\partial u}{\partial x_j}
$$
and we adopt Einstein convention for summation.

\section{Main Results \& discussion}

\subsection{Main results of problem \eqref{1v}}

We consider the following definition of weak solutions for problem \eqref{1v}.

\begin{defin}\label{definition1} 
We say that the function $v\in \mL^\infty(0,T;\mL^\infty(\TT))$ is a weak solution of \eqref{1v} with initial data $v_0$ if 
$$
\int_\TT \varphi(0,x)v_0(x)\,dx-\int_0^T\int_\TT \pat \varphi(t,x)v(t,x)+ \D^2\varphi(t,x)e^{-v(t,x)}\,dx\,dt=0,
$$
for all $\varphi\in \mW^{1,1}(0,T;\mL^1(\TT))\cap \mL^1(0,T;\mW^{4,1}(\TT))$.
\end{defin}

We prove that
\begin{theorem}\label{theorem1}
Let $v_0\in \mA^0(\TT)$ be such that the value
\begin{equation}\label{small1}
\de(|v_0|_0)=2-e^{|v_0|_0}-7|v_0|_0e^{|v_0|_0}-6|v_0|_0^2e^{|v_0|_0}-|v_0|_0^3e^{|v_0|_0}
\end{equation}
satisfies $0<\de(|v_0|_{0})$. Then, there exist at least one global weak solution in the sense of Definition \ref{definition1} to equation \eqref{1v} having the regularity
\begin{equation}\label{ex1}
v\in  \mL^\infty\left([0,T]\times \TT\right)\cap \mL^{\frac{4}{3}}\left(0,T;\mW^{3,\infty}(\TT)\right)\cap \mathcal{M}\left(0,T;\mW^{4,\infty}(\TT)\right)\cap \mL^2\left(0,T;{\mH}^2(\TT)\right)
\end{equation}
for any $T>0$. Furthermore, the solution satisfies
$$
v\in \mathscr{C}([0,T];\mL^2(\TT)),
$$
\begin{equation}\label{dec1A0}
\|v(t)\|_{\mL^\infty}\leq |v_0|_0e^{-\delta(|v_0|_0) t}\qq\forall t\in[0,T],
\end{equation}
for $\delta(|v_0|_0)$ defined in \eqref{small1}.
\end{theorem}

\begin{remark}
In order $v_0$ satisfies \eqref{small1}, it is enough to have
$$
|v_0|_0\leq 0.1.
$$
In particular, for $d=1$,
$$
v_0(x)=0.1\sin(1000 x)
$$
is a valid initial datum satisfying our hypothesis and having, at the same time, very large $\mH^s(\TT)$ norms.
\end{remark}

Furthermore, for the case where the initial data has certain Sobolev regularity, we have that
\begin{theorem}\label{theorem2}
Let $v_0\in \mA^0(\TT)\cap \mH^2(\TT)$ be such that \eqref{small1} is satisfied. 
Then, the solution constructed in Theorem \ref{theorem1} also satisfies
\begin{equation}\label{reg1}
v\in \mathscr{C}\left([0,T];{\mH}^2(\TT)\right)\cap \mL^2\left(0,T;{\mH}^4(\TT)\right).
\end{equation}
Furthermore, beyond \eqref{dec1A0}, the solution obeys
\begin{equation}\label{H2H4}
\|v(t)\|_{\mH^2}+c_1\int_0^T \|v(t)\|_{\mH^4}^2\,dt\leq c_2,
\end{equation}
\begin{equation}\label{expHr}
\|v(t)\|_{\mH^r}\leq c_3e^{-c_4 t} \quad \text{for all} \quad 0\leq r<2,
\end{equation}
where $c_i=c_i(v_0)$.
\end{theorem}

\subsection{Main results of problem \eqref{2v}}
First we state our definition of weak solution for problem \eqref{2v}:
\begin{defin}\label{def2}
A function $v\in \mL^\infty(0,T; \mL^\infty(\TT))$ is a weak solution of \eqref{2v} with initial data $v_0$ if
$$
\int_{\TT}v_0(x)\varphi(0,x)\,dx-\int_0^T\int_\TT \pat \varphi(t,x)v(t,x)+\,\frac{\D^2\varphi(t,x)}{(1+v(t,x))^3}\,dx\,dt=0 
$$
for all $\varphi\in C^\infty([0,T)\times \TT)$.
\end{defin}

We prove

\begin{theorem}\label{teo2ex}
Let $v_0\in  \mA^0(\TT)$ be such that the value
\begin{equation}\label{small2}
\de(|v_0|_0)=6-\frac{3}{(1-|v_0|_0)^{4}}\left( 1+28\frac{|v_0|_0}{(1-|v_0|_0)}+120\frac{|v_0|_0^2}{(1-|v_0|_0)^2}+120\frac{|v_0|_0^3}{(1-|v_0|_0)^3}
\right)
\end{equation}
satisfies $0<\de(|v_0|_{0})$ and
$$
|v_0|_0<1. 
$$
Then there exists at least one global weak solution in the sense of Definition \ref{def2} to equation \eqref{2v} having the regularity
\begin{equation}\label{ex2}
v\in  \mL^\infty\left([0,T]\times \TT\right)\cap \mL^{\frac{4}{3}}\left(0,T;\mW^{3,\infty}(\TT)\right)\cap \mathcal{M}\left(0,T;\mW^{4,\infty}(\TT)\right)\cap \mL^2\left(0,T;{\mH}^2(\TT)\right).
\end{equation}
Furthermore, the solution satisfies
$$
v\in \mathscr{C}([0,T];\mL^2(\TT)),
$$
\begin{equation}\label{decA0}
\|v(t)\|_{\mL^\infty}\leq |v_0|_0e^{-\delta(|v_0|_0) t}\qq\forall t\in[0,T],
\end{equation}
for $\de(|v_0|_0)$ satisfying \eqref{small2}.
\end{theorem}

\begin{remark}
In order $v_0$ satisfies \eqref{small2}, it is enough to have
$$
|v_0|_0\leq 0.023.
$$
\end{remark}

As for problem \eqref{1v}, for initial data having certain Sobolev regularity, we have that

\begin{theorem}\label{teo2reg}
Let $v_0\in  \mA^0(\TT)\cap  \mH^2(\TT)$ be a function satisfying the smallness condition in \eqref{small2} and
$$
|v_0|_0<1. 
$$
Then, the solution constructed in Theorem \ref{theorem2} also satisfies
\begin{equation}\label{reg1b}
v\in \mathscr{C}\left([0,T];{\mH}^2(\TT)\right)\cap \mL^2\left(0,T;{\mH}^4(\TT)\right).
\end{equation}
Furthermore, beyond \eqref{decA0}, the solution obeys
\begin{equation}\label{H2H4b}
\|v(t)\|_{\mH^2}+c_1\int_0^T \|v(t)\|_{\mH^4}^2\,dt\leq c_2,
\end{equation}
\begin{equation}\label{expHrb}
\|v(t)\|_{\mH^r}\leq c_3e^{-c_4 t} \quad \mbox{for all} \quad 0\leq r<2,
\end{equation}
where $c_i=c_i(v_0)$.
\end{theorem}

\subsection{Discussion}
In this paper we prove the global existence of weak solution for \eqref{1v} and \eqref{2v} for initial data satisfying a size restriction in the Wiener algebra. One of the main advantages is that the size restriction is explicit and concerns a lower order norm (it does not impose any requirement on the size of derivatives of $v$). Furthermore, when the domain is $\mathbb{R}^d$, the norm in the Wiener algebra is invariant by the scaling \eqref{scaling}. This makes the Wiener algebra a critical space for problems \eqref{1v} and \eqref{2v}.

In terms of the original problem \eqref{1u}, our results state that if the Laplacian of the initial data  satisfies certain explicit size restriction in the Wiener algebra, then the solution $u$ exists and its Laplacian remains globally bounded. In particular, the Laplacian of these solutions $u$ can not be a singular measure (compare with \cite{liu2016existence, gao2017gradient} and note that the $u$ reconstructed from our solutions $v$ satisfies $u(t)\in W^{2,\infty}(\TT)$). Indeed, in terms of the original variables, our existence results could be stated as follows

\begin{theorem}\label{theorem1b}
Let $u_0\in \mA^2(\TT)$ be such that the value
\begin{equation}\label{small1b}
\de(|u_0|_2)=2-e^{|u_0|_2}-7|u_0|_2e^{|u_0|_2}-6|u_0|_2^2e^{|u_0|_2}-|u_0|_2^3e^{|u_0|_2}
\end{equation}
satisfies $0<\de(|u_0|_{2})$. Then, there exist at least one global weak solution to equation \eqref{1u}
$$
u\in \mathscr{C}([0,T];\mH^2(\TT))\cap \mL^2\left(0,T;{\mH}^4(\TT)\right)\qq\forall 0<T<\infty.
$$
Furthermore, the solution satisfies
\begin{equation}\label{dec1A0b}
\|\Delta u(t)\|_{\mL^\infty}\leq |u_0|_2e^{-\delta(|u_0|_2) t}\qq\forall t,
\end{equation}
for $\delta(|v_0|_0)$ defined in \eqref{small1b} and a universal constant $c$.
\end{theorem}

and

\begin{theorem}\label{teo2exb}
Let $0\leq u_0\in  \mA^0(\TT)$ be such that the value $\de(|u_0^{-1}-1|_0)$ (defined as in \eqref{small2}) satisfies $0<\de(|u_0^{-1}-1|_{0})$ and
$$
|u_0^{-1}-1|_0<1,
$$
(in particular $0.5<u_0$). Then there exists at least one global weak solution 
$$
u^{-1}\in \mathscr{C}([0,T];\mL^2(\TT))\cap \mL^2\left(0,T;{\mH}^4(\TT)\right)\qq\forall 0<T<\infty,
$$
(in the sense that $v=u^{-1}-1$ satisfies Definition \ref{def2}). Furthermore
\begin{equation}\label{decA0b}
\|u^{-1}(t)-1\|_{\mL^\infty}\leq |u^{-1}_0-1|_0e^{-\delta(|u_0^{-1}-1|_0) t}\qq\forall t.
\end{equation}
In particular
$$
\limsup_{t\rightarrow\infty} \|u(t)\|_{L^\infty}=1.
$$
\end{theorem}

Concerning \eqref{2u}, our results imply that if the appropriate transformation of the initial data satisfies certain hypotheses, then the solution $u$ exists and remains bounded and positive (actually, $0.5<u(t)$). Furthermore, we can also quantify the convergence rate towards the steady state (compare with \cite{gao2017vicinal,gao2017weak}).

After the completion of this work, two new papers studying \eqref{1u} appeared, one by Jian-Guo Liu \& Robert Strain \cite{LS} and another by David Ambrose \cite{Ambrose}. 

Liu \& Strain \cite{LS} prove global existence, uniqueness, optimal large time decay rates, and uniform gain of analyticity for the exponential PDE \eqref{1u} (when the domain is the real space $\mathbb{R}^d$). We emphasize that the size restriction on the initial data in \cite{LS} ($|v_0|_0<0.1$) is equivalent to the size restriction in Remark 1 and our techniques are somehow close to the techniques in \cite{LS}. Obviously, the fact that the domain is the $d-$dimensional space instead of the $d-$dimensional torus as in our paper implies that the decay estimates in Liu \& Strain \cite{LS} are algebraic instead of exponential (as in our Theorem \ref{theorem1}). Interestingly, these authors also prove the instant gain of spatial analyticity and the uniqueness of solution.

Ambrose \cite{Ambrose} showed that in the case of the $d-$dimensional torus, the solutions become analytic at any positive time, with the radius of analyticity growing linearly for all time (in Liu \& Strain, due to the fact that they consider the free space, the growth was like $t^{0.25}$). Furthermore, using different estimates, Ambrose proved that the size restriction $|v_0|_0<0.25$ is enough to guarantee global existence, thus, improving the conditions in Theorem \ref{theorem1} and in \cite{LS}.

\section{The problem \eqref{1v}}
We consider the following \emph{weakly nonlinear} regularized analog of \eqref{1v}
\begin{equation}\label{pbv1reg}
\begin{cases}
\begin{array}{lll}
\ds \pat v_N=\D^2 \sum_{j=0}^N\frac{(-v_N)^j}{j!} & \text{ in }  [0,T]\times \TT,&\\
\ds v_N(0,x)=v_0(x)=\D u_0(x) &\text{ in }\TT. &
\end{array}
\end{cases}
\end{equation}
For this problem, we can construct a solution $v_N$ following a standard Galerkin approach. 

\subsection{Estimates for the regularized problem in the Wiener algebra $\mathscr{A}^0$}\label{pb1}

\begin{proposition}[Estimates in the Wiener algebra]\label{propo1}
Let $v_0\in \mA^0(\TT)$ be a function satisfying the condition in \eqref{small1}. Then, every approximating sequence of solutions $\{v_N\}_N$ of \eqref{pbv1reg} is uniformly bounded in
\[
\{v_N\}_N\in \mW^{1,1}(0,T;\mA^0(\TT))\cap \mL^1(0,T;{\mA}^4(\TT)).
\]
Furthermore, we have that
\begin{equation*} 
|v_N(t)|_0\leq |v_0|_0e^{-\delta(|v_0|_0) t}\qq\forall t\in[0,T],
\end{equation*}
with $\delta(|v_0|_0)$ satisfying \eqref{small1}. 
\end{proposition}

\begin{proof}To simplify the notation, we write $v$ when referring to $v_N$. We compute
\begin{align}
\pat v=&\D^2 \left( \sum_{j=0}^N(-1)^j\frac{v^j}{j!}  \right) \nonumber\\
=&\sum_{j= 1}^N\frac{(-1)^j}{(j-1)!}
\biggl[
v^{j-1}v,_{ii\ell\ell}
+2(j-1)v^{j-2}v,_{ii\ell}v,_\ell+(j-1)|v,_{ii}|^2v^{j-2}\nonumber\\
&
+2(j-1)(j-2)v^{j-3}v,_{ii}v,_\ell v,_\ell +4(j-1)(j-2)v^{j-3}v,_{i\ell}v,_{\ell}v,_i\nonumber\\
&+2(j-1)v^{j-2}(v,_{i\ell} v,_\ell),_i +(j-1)(j-2)(j-3)v^{j-4}v,_\ell v,_\ell v,_i v,_i
\biggr]\nonumber\\
=&\sum_{j=1}^N\frac{(-1)^j}{(j-1)!}
\biggl[
v^{j-1}v,_{ii\ell\ell}
+4(j-1)v^{j-2}v,_{ii\ell}v,_\ell+(j-1)|v,_{ii}|^2v^{j-2}\nonumber\\
&
+2(j-1)(j-2)v^{j-3}v,_{ii}v,_\ell v,_\ell +4(j-1)(j-2)v^{j-3}v,_{i\ell}v,_{\ell}v,_i \nonumber\\
&+2(j-1)v^{j-2}v,_{i\ell} v,_{i\ell}+(j-1)(j-2)(j-3)v^{j-4}v,_\ell v,_\ell v,_i v,_i
\biggr].\label{eq:v}
\end{align}
 
In Fourier variables  and omitting the time variable, the above equality reads
$$
\pat \fv(k)=-|k|^4\fv(k)+\sum_{r=1}^7\sum_{j= 2}^N\frac{(-1)^j}{(j-1)!}\,\fn^ j_r(k)
$$
where the nonlinearities are
\begin{align}
\fn_1^ j(k)=&
\sum_{\al^1\in\ZZ}\ldots\sum_{\al^{j-1}\in\ZZ}\fv(\al^{j-1})\prod_{n=1}^{j-2}\fv(\al^n-\al^{n+1}) |k-\al^{1}|^4\fv(k-\al^{1}),\nonumber\\ 
\fn_2^ j(k)=&4(j-1)\sum_{\al^1\in\ZZ}\ldots\sum_{\al^{j-1}\in\ZZ}\biggl[ \fv(\al^{j-1})\prod_{n=2}^{j-2}\fv(\al^ n-\al^ {n+1})\fv(\al^1-\al^2)(\al^1_{\ell}-\al^2_{\ell})|\al^1-\al^2|^2\nonumber \\
&\qq\qq\qq\qq\qq\qq\qq\qq\times(k_\ell-\al_\ell^{1})\fv(k-\al^{1})\biggr],\nonumber
\\ 
\fn_3^ j(k)=&(j-1)\sum_{\al^1\in\ZZ}\ldots\sum_{\al^{j-1}\in\ZZ}\fv(\al^{j-1})\prod_{n=2}^{j-2}\fv(\al^ n-\al^ {n+1})\fv(\al^1-\al^2)|\al^1-\al^2|^2 |k-\al^{1}|^2\fv(k-\al^{1}),\nonumber
\\ 
\fn_4^ j(k)=&2(j-1)(j-2)\sum_{\al^1\in\ZZ}\ldots\sum_{\al^{j-1}\in\ZZ}\biggl[\fv(\al^{j-1})\prod_{n=3}^{j-2}
\fv(\al^ n-\al^ {n+1})(\al^ 2_\ell-\al_\ell^{3})\fv(\al^2-\al^{3}) \nonumber\\
&\qq\qq\qq\qq\qq\qq\qq\qq\qq\q \times\fv(\al^1-\al^2)|\al^1-\al^2|^2(k_\ell-\al_\ell^{1})\fv(k-\al^{1})
\biggr]\,.\nonumber
\end{align}
Similarly, 
\begin{align}
\fn_5^ j(k)=&4(j-1)(j-2)\sum_{\al^1\in\ZZ}\ldots\sum_{\al^{j-1}\in\ZZ}\biggl[\fv(\al^{j-1})\prod_{n=3}^{j-2}
\fv(\al^ n-\al^ {n+1})
(\al_i^2-\al_i^3)(\al_\ell^2-\al_\ell^3)
\fv(\al^2-\al^3) 
\nonumber\\
&\qq\qq\qq\qq\qq\qq\qq\qq \qq\q\times 
(\al^1_\ell-\al^2_\ell)\fv(\al^1-\al^2) 
(k_i-\al^1_i)\fv(k-\al^1)\biggr],\nonumber
\\ 
\fn_6^ j(k)=&2(j-1)\sum_{\al^1\in\ZZ}\ldots\sum_{\al^{j-1}\in\ZZ}\biggl[\fv(\al^{j-1})\prod_{n=2}^{j-2}
\fv(\al^ n-\al^ {n+1})
(\al_i^1-\al_i^2)(\al_\ell^1-\al_\ell^2)
\fv(\al^1-\al^2) 
\nonumber\\
&\qq\qq\qq\qq\qq\qq\qq\qq \times 
(k_i-\al_i^1)(k_\ell-\al_\ell^1)
\fv(k-\al^1)\biggr],\nonumber
\\ 
\fn_7^ j(k)=&(j-1)(j-2)(j-3)\sum_{\al^1\in\ZZ}\ldots\sum_{\al^{j-1}\in\ZZ}\biggl[\fv(\al^{j-1})\prod_{n=4}^{j-2}
\fv(\al^ n-\al^ {n+1})
\nonumber\\
&\times(\al_i^3-\al_i^4)\fv(\al^3-\al^4) 
(\al_i^2-\al_i^3)\fv(\al^2-\al^3) 
(\al_\ell^1-\al_\ell^2)\fv(\al^1-\al^2) 
(k_\ell-\al_\ell^1)
\fv(k-\al^1)\biggr]. \nonumber
\end{align}
We want to obtain an estimate for $|v|_0$. Reintroducing the time variable and since
\begin{equation}\label{dert}
\pat |v(t,k)|=\frac{\Re(\overline{v}(t,k)\pat v(t,k))}{|v(t,k)|},
\end{equation}
we have that
$$
\frac{d}{dt}|v(t)|_0\leq -|v(t)|_4+\sum_{r=1}^7\sum_{j= 2}^\infty\frac{|\mathcal{N}^ j_r|_0(t)}{(j-1)!}.
$$
Then, using Tonelli's Theorem, we estimate $\mathcal{N}_1^j$ as
\begin{align}
|\mathcal{N}_1^ j(t)|_0\leq& \sum_{k\in\ZZ}\sum_{\al^1\in\ZZ}\ldots\sum_{\al^{j-1}\in\ZZ}\bigg{|}\fv(t,\al^{j-1})\prod_{n=1}^{j-2}\fv(t,\al^n-\al^{n+1}) |k-\al^{1}|^4\fv(t,k-\al^{1})\bigg{|}\nonumber\\
\leq &|v(t)|_4|v(t)|_0^{j-1}.\nonumber
\end{align}
Using the trivial inequality $|\beta_\ell|\leq \sqrt{\beta_1^2+\dots +\beta_\ell^2+\dots+\beta_d^2}=|\beta|,$ we have that
\begin{align}
|\mathcal{N}_2^ j(t)|_0\leq& 4(j-1)|v(t)|_0^{j-2}|v(t)|_3|v(t)|_1,\nonumber\\
|\mathcal{N}_3^ j(t)|_0\leq& (j-1)|v(t)|_0^{j-2}|v(t)|_2^2,\nonumber\\
|\mathcal{N}_4^ j(t)|_0\leq& 2(j-1)(j-2)|v(t)|_0^{j-3}|v(t)|_2|v(t)|_1^2,\nonumber\\
|\mathcal{N}_5^ j(t)|_0\leq& 4(j-1)(j-2)|v(t)|_0^{j-3}|v(t)|_2|v(t)|_1^2,\nonumber\\
|\mathcal{N}_6^ j(t)|_0\leq& 2(j-1)|v(t)|_0^{j-2}|v(t)|_2^2,\nonumber\\
|\mathcal{N}_7^j(t)|_0\leq & (j-1)(j-2)(j-3)|v(t)|_0^{j-4}|v(t)|_1^4.\nonumber
\end{align}
The interpolation inequality \eqref{interpolation} provides us with
$$
|v(t)|_3|v(t)|_1+|v(t)|_2^2\leq 2|v(t)|_4|v(t)|_0,\quad |v(t)|_2|v(t)|_1^2\leq |v(t)|_0^2|v(t)|_4,\quad |v(t)|_1^4\leq |v(t)|_0^3|v(t)|_4.
$$
Thus,
\begin{equation*}
\sum_{\ell=1}^7\sum_{j\ge 2}\frac{1}{(j-1)!}|\mathcal{N}^ j_\ell(t)|_0\leq \sum_{j\ge 2}\frac{|v(t)|_4|v(t)|_0^{j-1}}{(j-1)!}
\left(
1+(j-1)\left[7+6(j-2)+(j-2)(j-3)\right]
\right).
\end{equation*}
We obtain that
\begin{equation*}
\frac{d}{dt}|v(t)|_0\leq |v(t)|_4\left[-1+\sum_{j\ge 2}\frac{|v(t)|_0^{j-1}}{(j-1)!}\left(
1+(j-1)\left[7+6(j-2)+(j-2)(j-3)\right]
\right)\right].
\end{equation*}
We compute
\begin{align}
\sum_{j\ge 2}\frac{|v(t)|_0^{j-1}}{(j-1)!}&=e^{|v(t)|_0}-1,\qquad
\sum_{j\ge 2}\frac{|v(t)|_0^{j-1}}{(j-2)!}=|v(t)|_0\,e^{|v(t)|_0},\nonumber\\
\sum_{j\ge 3}\frac{|v(t)|_0^{j-1}}{(j-3)!}&=|v(t)|_0^2\,e^{|v(t)|_0},\qquad
\sum_{j\ge 4}\frac{|v(t)|_0^{j-1}}{(j-4)!}=|v(t)|_0^3\,e^{|v(t)|_0}.\nonumber
\end{align}
As a consequence, we find that
$$
\frac{d}{dt}|v(t)|_0\leq |v(t)|_4\left[e^{|v(t)|_0}+7|v(t)|_0e^{|v(t)|_0}+6|v(t)|_0^2e^{|v(t)|_0}+|v(t)|_0^3e^{|v(t)|_0}-2\right].
$$
Let 
\[
\de(|v(t)|_0)=2-e^{|v(t)|_0}-7|v(t)|_0e^{|v(t)|_0}-6|v(t)|_0^2e^{|v(t)|_0}-|v(t)|_0^3e^{|v(t)|_0}.
\]
As $v_0$ satisfies 
$$
e^{|v_0|_0}-7|v_0|_0e^{|v_0|_0}-6|v_0|_0^2e^{|v_0|_0}-|v_0|_0^3e^{|v_0|_0}<2,
$$ 
we have that
$$
|v(t)|_0\leq |v_0|_0
$$
for $0\leq t\leq t^*$. To prove that $t^*=\infty$ we argue by contradiction. In the case where $t^*<\infty$, necessarily we have that
$$
|v(t^*)|_0=|v_0|_0,
$$
but then
$$
e^{|v(t^*)|_0}-7|v(t^*)|_0e^{|v(t^*)|_0}-6|v(t^*)|_0^2e^{|v(t^*)|_0}-|v(t^*)|_0^3e^{|v(t^*)|_0}<2
$$ 
and
$$
\frac{d}{dt}|v(t)|_0\bigg{|}_{t=t^*}<0,
$$
which is a contradiction with $t^*$ begin finite.

Then we have that
$$
|v(t)|_0+\delta(|v_0|_0)\int_0^t|v(s)|_4\,ds\leq |v_0|_0,
$$
where $0<\delta(|v_0|_0)$. Furthermore, using a standard Poincar\'e-like inequality, we find that 
$$
|v(t)|_0\leq |v_0|_0e^{-\delta(|v_0|_0) t}\qq\forall t\in [0,T].
$$
The inequality
\begin{equation*}
|\pat v(t)|_0\le C|v(t)|_4,
\end{equation*}
guarantees that $v_N\in \mW^{1,1}(0,T; \mA^0(\TT))$
\end{proof}

\subsection{Convergence of the approximate problems}

\begin{proposition}[Compactness results]\label{propo2}
Let $v_0\in \mA^0(\TT)$ be a function satisfying the condition in \eqref{small1} and $\partial$ be any differential operator of order one. Then, up to subsequences, we  have that every approximating sequence of solutions $\{v_N\}_N$ of \eqref{pbv1reg} verifies
\begin{align} 
\{ v_N\}_N\overset{*}{\rightharpoonup}  v\q &\text{in} \q \mL^{\infty}(0,T;\mL^\infty(\TT)),\label{linflinf}\\
\{ \partial^4 v_N\}_N\overset{*}{\rightharpoonup}  \partial^4 v\q& \text{in} \q \mathcal{M}(0,T;\mL^\infty(\TT)),\label{MW4}\\
\{\partial^3 v_N\}_N\overset{*}{\rightharpoonup}  \partial^3 v\q& \text{in} \q \mL^{\frac{4}{3}}(0,T;\mL^{\infty}(\TT)),\label{43}\\
\{ v_N\}_N\rightharpoonup v\q &\text{in} \q \mL^2(0,T,\mH^2(\TT)),\label{L2H2}\\
\{ v_N\}_N\rightarrow v\q &\text{in}\q\mL^2(0,T;\mH^r(\TT))\q\text{with}\q0\leq r<2. \label{L2Hr1}
\end{align}
Furthermore,
$$
v\in \mathscr{C}([0,T];\mL^2(\TT)).
$$
\end{proposition}

\begin{proof}
We need to prove the compactness of the sequence of approximate solutions in appropriate spaces. Thanks to Proposition \ref{propo1}, the approximate solutions are uniformly bounded in 
\begin{equation}\label{boundedness}
\{ v_N\}_N\in \mL^\infty(0,T;\mA^0(\TT))\cap \mL^1(0,T;\mA^4(\TT)).
\end{equation}
Due to Banach-Alaoglu Theorem, this boundedness is enough to have \eqref{linflinf}.\\
Similarly, we have the uniform bound 
$$
\{ \partial^4 v_N\}_N\in \mathcal{M}(0,T;\mL^\infty(\TT)),
$$
where $\partial^4$ is a differential operator of order fourth. Using Banach-Alaoglu Theorem, we find \eqref{MW4}.\\
The interpolation inequality in \eqref{interpolation} gives us 
\begin{align}
\int_0^T|v_N(t)|_r^{\frac{4}{r}}\,dt&\le \int_0^T|v_N(t)|_0^{(4-r)\frac{1}{r}}|v_N(t)|_4\,dt\nonumber\\
&\le \sup_{t\in [0,T]}|v_N(t)|_0^{(4-r)\frac{1}{r}} \|v_N\|_{\mL^1(0,T;\mA^4(\TT))}\nonumber\\
&<c(v_0,r),\;\;0<r<4\,,\label{aux}
\end{align}
so \eqref{43} follows.\\
Furthermore, starting from \eqref{boundedness}, we will show
\begin{equation}\label{unif1}
\{ v_N\}_N\q\text{is uniformly bounded in}\q   \mL^2(0,T;\mH^{2}(\TT)),
\end{equation}
\begin{equation}\label{unif2}
\{\pat v_N\}_N\q\text{is uniformly bounded in}\q   \mL^2(0,T;\mH^{-2}(\TT)).
\end{equation}

Let us begin with the proof of \eqref{unif1}. Using the inequality \eqref{aux} with $r=2$ and the finiteness of the domain, we have that
\begin{equation}\label{regH2}
\int_0^T \|v_N(t)\|_{\mH^2}^2\,dt<c(v_0),
\end{equation}
concluding \eqref{unif1}. As far as \eqref{unif2} is concerned, being
\[
\|\pat v_N(t)\|_{\mH^{-2}}=\sup_{
\footnotesize
\begin{array}{c}
\vp\in \mH^2(\TT)\\ \|\vp\|_{\mH^2}\le 1
\end{array}
}\bigg{|}\langle\pat v_N(t),\vp\rangle\bigg{|},
\]
then we have
\begin{align}
\left|\int_{\TT}\pat v_N(t)\vp\,dx\right|
&=\left|\int_{\TT}\D \left(\sum_{j=0}^N\frac{(-v_N(t))^j}{j!}\right) \D\vp\,dx\right|\nonumber\\
&\le \int_{\TT}e^{|v_N(t)|}|\D v_N(t)||\D\vp|\,dx+\int_{\TT} e^{|v_N(t)|}|\N v_N(t)|^2|\D\vp|\,dx\,.\nonumber
\end{align}
We compute that
\begin{align}
\int_{\TT}e^{|v_N(t)|}|\D\vp|\left[|\D v_N(t)|+ |\N v_N(t)|^2\right]\,dx & \le e^{|v_N(t)|_{0}}
(\|v_N(t)\|_{\mH^2}+\|\nabla v_N(t)\|_{\mL^4}^2)\|\vp\|_{\mH^2} \nonumber \\
&\le C(1+|v_N(t)|_0)e^{|v_N(t)|_{0}} \|v_N(t)\|_{\mH^2},
\end{align}
where we have used the inequality 
\begin{equation}\label{smart1}
\|\N v(t)\|_{\mL^4}^2\le c\|v(t)\|_{ \mH^{2}}|v(t)|_0.
\end{equation}
The above inequality can be easily proved using integration by parts and H\"{o}lder's inequality and, thus, we left it for the interested reader.

As a consequence, we obtain that
\begin{align}
\|\pat v_N(t)\|_{\mH^{-2}}^2&\leq\sup_{
\footnotesize
\begin{array}{c}
\vp\in \mH^2(\TT)\\ \|\vp\|_{\mH^2}\le 1
\end{array}
}\left|\int_{\TT}\pat v_N(t)\vp\,dx\right|^2\nonumber\\
&\leq c e^{2\|v_N\|_{\mL^\infty(\mA^0)}} (1+\|v_N\|_{\mL^\infty(\mA^0)})^2\,\|v_N(t)\|_{\mH^2}^2,\nonumber
\end{align}
and the boundedness of $\int_0^T\|\pat v_N(t)\|_{\mH^{-2}}^2\,dt$ follows thanks to \eqref{regH2}. The fact that  $v\in \mathscr{C}([0,T];\mL^2(\TT))$ follows straightforwardly from \eqref{regH2} and the finiteness of $\int_0^T\|\pat v(t)\|_{\mH^{-2}}^2\,dt$. Then, up to a subsequence, we have both \eqref{L2H2} and
$$
\{ \pat v_N\}_N\rightharpoonup \pat v \qquad \text{ in\q $\mL^2(0,T;\mH^{-2}(\TT))$}\,.
$$
Invoking \cite[Corollary 4]{S}, with the choice
$$
X_0= \mH^2(\TT),\quad X=\mL^2(\TT),\quad X_1= \mH^{-2}(\TT),
$$
we obtain that
$$
\{ v_N\}_N\rightarrow v\qquad \text{ in}\q\mL^2(0,T;\mL^2(\TT)).
$$

Using interpolation in Sobolev spaces and the uniform boundedness in $\mL^2(0,T;\mH^2(\TT))$, we have that
\begin{align}
\int_{0}^T\|v_N(t)-v(t)\|_{\mH^r}^2\,dt&\leq \int_{0}^T\|v_N(t)-v(t)\|_{\mL^2}^{2-r}\|v_N(t)-v(t)\|_{\mH^2}^{r}\,dt\nonumber\\
&\le\left(
\int_{0}^T\|v_N(t)-v(t)\|_{\mL^2}^2\,dt
\right)^\frac{2-r}{2}
\left(
\int_{0}^T\|v_N(t)-v(t)\|_{\mH^2}^2\,dt
\right)^\frac{r}{2}\nonumber
\end{align}
for every $0\leq r<2$. Thus \eqref{L2Hr1} follows.
\end{proof}

\subsection{Proof of Theorem \ref{theorem1}}\label{pttl1}

We have to pass to the limit in $N$ to conclude the existence of the weak solution. First, let us remark that
$$
\int_0^T\int_\TT  \left(\sum_{j=0}^N\frac{(-v_N)^j}{j!}-e^{-v_N}\right)\D^2\varphi \,dx\,dt\rightarrow 0
$$
as $N\rightarrow \infty$. We note that
$$
e^{x}-e^{y}=\int_0^1 \partial_{\lambda}e^{\lambda x+(1-\lambda)y }\,d\lambda=\int_0^1 e^{\lambda x+(1-\lambda)y }(x-y)\,d\lambda.
$$
Using  Propositions \ref{propo1} and \ref{propo2}, together with the previous equality, we have that
$$
\left|\int_{\TT}(e^{v_N(t)}-e^{v(t)})\Delta^2\varphi \,dx\right|\leq e^{\|v_N\|_{\mL^\infty(\mL^\infty)}+\|v\|_{\mL^\infty(\mL^\infty)}}\|v_N(t)-v(t)\|_{\mL^2}\|\Delta^2\varphi(t)\|_{\mL^2}
$$
for every $\vp\in \mW^{1,1}(0,T;\mL^1(\TT))\cap\mL^1(0,T;\mW^{4,1}(\TT))$.
Thus,
$$
\lim_{N\rightarrow \infty}\left|\int_0^T\int_{\TT}(e^{v_N}-e^{v})\Delta^2\varphi \,dx\,dt\right|=0,
$$
and we can pass to the limit in the weak formulation. The regularity in \eqref{ex1} follows from Proposition \ref{propo2} as well. Finally, the weakly-$*$ lower semicontinuity of the norm guarantees that
$$
\|v(t)\|_{\mL^\infty}\leq \lim_{N\rightarrow\infty} \|v_N(t)\|_{\mL^\infty}\leq |v_0|_{0}e^{-\delta(|v_0|_0) t},
$$
so we recover \eqref{dec1A0}.

\subsection{Estimates for the regularized problem in the Sobolev space $\mathscr{H}^2$}
\begin{proposition}[Estimates in Sobolev spaces]\label{propo4}
Let $v_0\in \mA^0(\TT)\cap \mH^2(\TT)$ be a function satisfying the condition in \eqref{small1}. Then, every approximating sequence of solutions $\{v_N\}_N$ of \eqref{pbv1reg} is uniformly bounded in
\[
\{v_N\}_N\in \mL^\infty(0,T;\mH^2(\TT))\cap \mL^2(0,T;{\mH}^4(\TT)).
\]
Furthermore, we have that
\begin{equation*} 
\|v_N(t)\|_{\mH^2}\leq C(|v_0|_0)\qq\forall t\in[0,T].
\end{equation*}
\end{proposition}

\begin{proof}[Sketch of the proof]
Let us briefly sketch the idea of the proof. First we multiply the equation by $\Delta^2 v$ and integrate by parts. This step, due to the nonlinearity of the equation, will create many different terms. These terms can be grouped in three categories: low (having a bi-laplacian times terms with 1 or 2 derivatives), medium (having a bi-laplacian times a term with three derivatives) and high order (having the bi-laplacian squared). The high order terms containing $|\Delta^2 v|^2$ will be absorbed by the linear terms due to the size restriction of the initial data. The low order terms (see $I_2, I_3, I_4, I_5$ and $I_6$ below) akin to
$$
\int |v|^n|\partial^2 v|^2\Delta^2v
$$
will be estimated as
$$
|v|_0^n\|v\|_{\mH^2}^2|v|_4.
$$
Then, taking advantage of $\int_0^t|v(s)|_4ds<\infty$ and using Gronwall inequality, we bound their contribution. The only remaining terms are the \emph{medium order} terms 
$$
\int v^n v,_j v,_{iij} v,_{\ell\ell r r}dx.
$$
(see $I_1$ below). For these terms we have to integrate by parts (several times) to find a perfect derivative hidden in $I_1$. Once this structure of perfect derivative is obtained, we can integrate by parts and obtain terms like
$$
\int\partial(v^n\partial v)|\partial^3 v|^2.
$$
Then we can conclude by noticing that
$$
\int\partial(v^n\partial v)|\partial^3 v|^2\leq \|v\|_{\mH^2}|v|_0^n\|\partial^3 v\|_{\mL^4}^2\leq C\|v\|_{\mH^2}^2|v|_0^n|\partial^2 v|_{2}.
$$
\end{proof}

\begin{proof}[Proof of Proposition \ref{propo4}]
We omit the time variable when not needed. We consider the solutions of \eqref{pbv1reg}. We want to obtain appropriate bounds on Sobolev spaces such that we can pass to the limit in $N$. As before, to simplify the notation, we write $v$ instead $v_N$. We multiply the equation in \eqref{pbv1reg} by $\D^2 v_N$ and integrate over $(0,T)\times \TT$, obtaining
\begin{align}
\frac{1}{2}\frac{d}{dt}\int_\TT |\D v|^2\,dx&=\int_\TT \D^2 \left(\sum_{n=0}^N\frac{(-v)^n}{n!}\right)\D^2v\,dx\nonumber\\
&=
-\int_\TT \sum_{n=1}^N\frac{(-v)^{n-1}}{(n-1)!}
|v,_{\ell\ell rr}|^2\,dx\nonumber\\
&\q+
\int_\TT \sum_{n=2}^N\frac{1}{(n-2)!}
\Bigl\{
(-v)^{n-2}\left[ 4v,_{iij}v,_j+v,_{ii}v,_{jj}+2v,_{ij}v,_{ij} \right]\nonumber\\
& \qq\qq\qq
-2(n-2)(-v)^{n-3}\left[ v,_{ii}v,_jv,_j +2v,_{ij}v,_jv,_i \right]
\nonumber\\
&\qq\qq\qq + (n-2)(n-3)(-v)^{n-4}v,_jv,_jv,_i v,_i\Bigr\}
v,_{\ell\ell rr}\,dx
\nonumber\\
&= 
-\int_\TT \sum_{n=1}^N\frac{(-v)^{n-1}}{(n-1)!}
|v,_{\ell\ell rr}|^2\,dx+\sum_{m=1}^6 I_m\,.\label{H2H4n}
\end{align}

We have that
\[
\forall\eps>0\q\exists \bar{N}\in\mathbb{N}:\q \forall N\ge \bar{N}\Rightarrow\q\left|  \sum_{n=1}^N\frac{(-v)^{n-1}}{(n-1)!}-e^{-v} \right|<\eps,
\]
thus, if $N$ is chosen large enough,
\[
-\int_\TT \sum_{n=1}^N\frac{(-v)^{n-1}}{(n-1)!}
|v,_{\ell\ell rr}|^2\,dx<0.
\]
We begin dealing with the $I_m$ terms. It holds 
\begin{align}
I_2+I_3&=\int_\TT \left(\sum_{n=2}^N\frac{(-v)^{n-2}}{(n-2)!}\right) \left(v,_{ii}v,_{jj}+2v,_{ij}v,_{ij}\right)v,_{\ell\ell rr}\,dx\nonumber\\
&\le ce^{|v|_0}|v|_4\|v\|_{ \mH^{2}}^2,\nonumber
\\
I_4+I_5&=-2\int_\TT \left(\sum_{n=3}^N\frac{(-v)^{n-3}}{(n-3)!}\right)\left(v,_{ii}v,_jv,_j+2v,_{ij}v,_jv,_i\right)v,_{\ell\ell rr}\,dx\nonumber\\
&\le c|v|_0 e^{|v|_0}|v|_4\|v\|_{ \mH^{2}}^2,\nonumber 
\\
I_6&=\int_\TT \left(\sum_{n=4}^N\frac{(-v)^{n-4}}{(n-4)!}\right)v,_jv,_jv,_\ell v,_\ell v,_{\ell\ell rr}\,dx\nonumber\\
&\le  c|v|_0 e^{|v|_0}|v|_4\|v\|_{ \mH^{2}}^2,\nonumber
\end{align}
where we have used the inequality \eqref{smart1}. So far, we have that $$\sum_{m=2}^6 I_m\le c\left(e^{|v|_0}+ |v|_0 e^{|v|_0} \right)|v|_4\|v\|_{ \mH^{2}}^2.$$ We are left with
\[
\begin{split}
I_1&=4\int_\TT \left(\sum_{n=2}^N\frac{(-v)^{n-2}}{(n-2)!}\right)v,_jv,_{iij} 
v,_{\ell\ell rr}\,dx\,.
\end{split}
\]

These terms are the \emph{medium} order terms mentioned in the sketch of the proof above. Our goal here is to manipulate the previous expression until we find a perfect $j-$derivative hidden in $I_1$ (a term like $\partial_j((\partial_i\partial_\ell\partial_r v)^2)$). Then integrating by parts once again we can charge in the low order elements (those with at most 1 derivative) and find integrals containing two elements having three derivatives times a number of low order terms. We observe that, in order the perfect derivative structure arises, we have to integrate by parts three times (in $i$, $\ell$ and $r$).

With this goal in mind we can start a rather tedious but elementary integration by parts,
\begin{align}
\frac{I_1}{4}&=-\int_\TT 
\left(\left(\sum_{n=2}^N\frac{(-v)^{n-2}}{(n-2)!}\right)v,_jv,_{iij}\right),_rv,_{\ell\ell r}\,dx\nonumber\\
&=\int_\TT 
\left[v,_r\left(\sum_{n=3}^N\frac{(-v)^{n-3}}{(n-3)!}\right)
v,_jv,_{iij}-
\left(\sum_{n=2}^N\frac{(-v)^{n-2}}{(n-2)!}\right)v,_{jr}v,_{iij}\right. \nonumber\\
&\qq\qq\left.-\left(\sum_{n=2}^N\frac{(-v)^{n-2}}{(n-2)!}\right)v,_{j}v,_{iijr}\right
]v,_{\ell\ell r}\,dx\nonumber.
\end{align}
We observe that we are left another two integration by parts so we can move an $i-$derivative between $v,_{iijr}$ and $v,_{\ell\ell r}$ and, after that, a $\ell-$derivative between $v,_{i\ell\ell r}$ and $v,_{ijr}$. We continue moving the $i-$derivative:
\begin{align}
\frac{I_1}{4}&=\int_\TT 
\left[
v,_r\left(\sum_{n=3}^N\frac{(-v)^{n-3}}{(n-3)!}\right)
v,_jv,_{iij}-
\left(\sum_{n=2}^N\frac{(-v)^{n-2}}{(n-2)!}\right)v,_{jr}v,_{iij}\right]
v,_{\ell\ell r}\,dx\nonumber\\
&\q+\int_\TT v,_{ijr}\left(\left(\sum_{n=2}^N\frac{(-v)^{n-2}}{(n-2)!}\right)v,_{j}v,_{\ell\ell r}\right),_i
\,dx\nonumber.
\end{align}
Now we move the $\ell-$derivative:
\begin{align}
\frac{I_1}{4}&=\int_\TT 
\left[
v,_r\left(\sum_{n=3}^N\frac{(-v)^{n-3}}{(n-3)!}\right)
v,_jv,_{iij}-
\left(\sum_{n=2}^N\frac{(-v)^{n-2}}{(n-2)!}\right)v,_{jr}v,_{iij}\right]v,_{\ell\ell r}\,dx\nonumber\\
&\q+\int_\TT v,_{ijr}\left(\left(\sum_{n=2}^N\frac{(-v)^{n-2}}{(n-2)!}\right)v,_{j}\right),_iv,_{\ell\ell r}
\,dx \nonumber\\
&\q -\int_\TT v,_{i\ell r}\left(\left(\sum_{n=2}^N\frac{(-v)^{n-2}}{(n-2)!}\right)v,_{j}v,_{ijr}\right),_\ell
\,dx\nonumber.
\end{align}
Expanding the latter expression, we find that
\begin{align}
\frac{I_1}{4}&=\int_\TT 
\left[
v,_r\left(\sum_{n=3}^N\frac{(-v)^{n-3}}{(n-3)!}\right)
v,_jv,_{iij}-
\left(\sum_{n=2}^N\frac{(-v)^{n-2}}{(n-2)!}\right)v,_{jr}v,_{iij}\right]v,_{\ell\ell r}\,dx\nonumber\\
&\q+\int_\TT v,_{ijr}\left(\left(\sum_{n=2}^N\frac{(-v)^{n-2}}{(n-2)!}\right)v,_{j}\right),_iv,_{\ell\ell r}
\,dx \nonumber\\
&\q -\int_\TT v,_{i\ell r}v,_{ijr}\left(\left(\sum_{n=2}^N\frac{(-v)^{n-2}}{(n-2)!}\right)v,_{j}\right),_\ell
\,dx\nonumber\\
&\q-\int_\TT \left(\sum_{n=2}^N\frac{(-v)^{n-2}}{(n-2)!}\right)v,_{j}v,_{i\ell r}v,_{ijr\ell} 
\,dx\nonumber.\end{align}
We note the perfect $j-$derivative structure present in the last term in the previous expression, and exploit that in a last integration by parts:
\begin{align}
I_1&=4\int_\TT 
\left[v,_r\left(
\sum_{n=3}^N\frac{(-v)^{n-3}}{(n-3)!}\right)
v,_jv,_{iij}-
\left(\sum_{n=2}^N\frac{(-v)^{n-2}}{(n-2)!}\right)v,_{jr}v,_{iij}\right]v,_{\ell\ell r}\,dx\nonumber\\
&\quad+4\int_\TT v,_{ijr}\left(\left(\sum_{n=2}^N\frac{(-v)^{n-2}}{(n-2)!}\right)v,_{j}\right),_iv,_{\ell\ell r}
\,dx\nonumber\\
&\q-4\int_\TT v,_{i\ell r}v,_{ijr}\left(\left(\sum_{n=2}^N\frac{(-v)^{n-2}}{(n-2)!}\right)v,_{j}\right),_\ell 
\,dx\nonumber\\
&\quad +2\int_\TT \left(\left(\sum_{n=2}^N\frac{(-v)^{n-2}}{(n-2)!}\right)v,_{j}\right),_jv,_{i\ell r}v,_{ir\ell}
\,dx.\nonumber
\end{align}
We recall the following inequality
\begin{equation}\label{smart2}
\|\N v(t)\|_{\mL^4}^2\le c\|v(t)\|_{\mL^2}|v(t)|_2.
\end{equation}
Now we can estimate this terms using \eqref{smart1} and \eqref{smart2} as
$$
I_1\le ce^{|v|_0}|v|_4\|v\|_{ \mH^2}^2\,.
$$
Putting all together, recovering the $N$ and the $t$ in the notation and using  Theorem \ref{theorem1}, we have that
\[
\esssup_{0\leq t\leq T}\|v_N(t)\|_{\mH^2}^2\leq\|v_0\|_{\mH^2}^2e^{c(|v_0|_0)\int_0^T|v_N(t)|_4\,dt}\leq C(|v_0|_0).
\]
Finally,taking $0<\eps\ll1$ small enough, we have that
\[
\sum_{n=1}^N\frac{(-v_N)^{n-1}}{(n-1)!}-e^{-v_N}+e^{-v_N}>e^{-v_N}-\eps>0
\]
so
\begin{align}
\left(e^{-\|v_N\|_{\mL^\infty(\mL^\infty)}}-\eps\right)
\int_0^T\int_\TT 
|(v_N),_{\ell\ell rr}|^2\,dx\,dt 
&\le
\int_0^T\int_\TT \sum_{n=1}^N\frac{(-v_N)^{n-1}}{(n-1)!}
|(v_N),_{\ell\ell rr}|^2\,dx\,dt\nonumber\\
&\le \int_0^T\sum_{m=1}^6 I_m(t)\,dt\nonumber\\
&\le c.\nonumber
\end{align}
Thus, we deduce that $v_N\in \mL^2(0,T;\mH^4(\TT))$. 
\end{proof}

\subsection{Convergence of the approximate problems}
 
\begin{proposition}[Compactness results]\label{propo5}
Let $v_0\in \mA^0(\TT)\cap \mH^2(\TT)$ be a function satisfying the condition in \eqref{small1} and $\partial$ be any differential operator of order one. Then, up to subsequences, we  have that every approximating sequence of solutions $\{v_N\}_N$ of \eqref{pbv1reg} verifies
\begin{equation}\label{L2Hr}
\{ v_N\}_N\rightarrow v\q \text{ in}\q\mL^2(0,T;\mH^r(\TT)),\;0\leq r<4.
\end{equation}
Furthermore, the limit function $v$ satisfies
$$
v\in\mathscr{C}([0,T];\mH^2(\TT)).
$$
\end{proposition}

\begin{proof} The proof of \eqref{L2Hr} is similar to Proposition \ref{propo2}. The continuity is obtained from the fact that
$$
\Delta v\in\mL^2(0,T;\mH^2(\TT)),\q\partial_t\Delta v\in\mL^2(0,T;\mH^{-2}(\TT)).
$$
\end{proof}

\subsection{Proof of Theorem \ref{theorem2}}\label{proofreg1}
The regularity \eqref{reg1} follows from Propositions \ref{propo4} and  \ref{propo5} and we are left with the proof of \eqref{H2H4} and \eqref{expHr}. The equality in \eqref{H2H4n} and the bound of the $I_m(t)$ terms in Proposition \ref{propo4} provide us with
\begin{equation*}
\frac{1}{2}\| v_N(t)\|^2_{\mH^2}+e^{-\|v\|_{\mL^\infty(\mL^\infty)}}\int_0^T\|v_N(t)\|_{\mH^4}^2\,dt\le c.
\end{equation*}

Then, the limit procedure can be achieved as in the proof of Theorem \ref{theorem1} and \eqref{H2H4} follows. Finally, the decay in \eqref{expHr} can be obtained using interpolation in Sobolev spaces
$$
\|v(t)\|_{\mH^r}\leq \liminf_{N\rightarrow\infty}\|v_N(t)\|_{\mH^r}\leq C(|v_0|_{0})e^{-\delta \left(1-\frac{r}{2}\right) t},\;0\leq r\leq 2.
$$

\section{The problem \eqref{2v}}\label{pb2}
We recall the definition of the binomial coefficient:
\[
\binom{n}{k}=\frac{n!}{k!(n-k)!}
\]
and compute
$$
\frac{1}{(1+v)^3}=\sum_{j= 0}^\infty(-1)^j\binom{j+2}{j}v^j=1-3v+\sum_{j=2}^\infty(-1)^j\binom{j+2}{j}v^{j}.
$$
Then, the equation \eqref{2v} becomes
\begin{equation*}
\pat v
=-3\D^2 v+\sum_{j=2}^{\infty}(-1)^j\binom{j+2}{j}\D^2v^{j} .
\end{equation*}
The approximating problem we consider is the following:
\begin{equation}\label{pbv2reg}
\begin{cases}
\begin{array}{lll}
\ds \pat v_N= \sum_{j= 0}^N(-1)^j\binom{j+2}{j}\D^2v_N^{j}  & \text{ in } [0,T]\times\TT, &\\
\ds v_N(0,x)= v_0(x)=\frac{1}{u_0(x)}-1 &\text{ in }\TT.&
\end{array}
\end{cases}
\end{equation}
For this problem, we can construct a solution $v_N$ following a standard Galerkin approach (one can also use the mollifier approach).

\subsection{Estimates for the regularized problem in the Wiener algebra $\mathscr{A}^0$}
 \begin{proposition}[Estimates in the Wiener algebra]\label{sap2}
Let $v_0\in  \mA^0(\TT)$ be a function satisfying the condition in \eqref{small2}. Then, every approximating sequence of solutions $\{v_N\}_N$ of \eqref{pbv2reg} is uniformly bounded in
\[
 \mW^{1,1}(0,T; \mA^0(\TT))\cap \mL^1(0,T;\mathscr{{\mA}}^4(\TT)).
\]
Furthermore, we have that
\begin{equation}\label{decexp2}
|v_N(t)|_0\leq |v_0|_0e^{- \delta(|v_0|_0) t}\qq\forall t\in[0,T]
\end{equation}
with $\delta(|v_0|_0)$ defined in \eqref{small2}. 
\end{proposition}
\begin{proof}
To simplify the notation, we write $v$ instead of $v_N$. We rewrite $\D^2 v^{j}$ (computed in \eqref{eq:v}) as
\begin{align}
\D^2v^j
&=\binom{j}{j-1}v^{j-1}v,_{ii\ell\ell}+2\binom{j}{j-2}v^{j-2}\left[|v,_{ii}|^2+4v,_{ii\ell}v,_\ell+2v,_{i\ell}v,_{i\ell}\right]\nonumber\\
&\q+12\binom{j}{j-3}v^{j-3}\left[ 2v,_{i\ell}v,_{\ell}v,_i+v,_\ell v,_\ell v,_{ii} \right] +24\binom{j}{j-4}v^{j-4}v,_\ell v,_\ell v,_i v,_i\nonumber
\end{align}
and set
\begin{equation}\label{eq2}
\pat v=-3\D^2 v+
\sum_{j= 2}^N(-1)^j\sum_{m=1}^7 \mathcal{N}^j_m 
\end{equation}
where
\begin{align}
\mathcal{N}^j_1&=\binom{j+2}{j}\binom{j}{j-1}v^{j-1}v,_{ii\ell\ell}=3\binom{j+2}{j-1}v^{j-1}v,_{ii\ell\ell},\nonumber
\\
\mathcal{N}^j_2&=2\binom{j+2}{j}\binom{j}{j-2}v^{j-2} |v,_{ii}|^2=12\binom{j+2}{j-2}v^{j-2} |v,_{ii}|^2,\nonumber
\\
\mathcal{N}^j_3&=8\binom{j+2}{j}\binom{j}{j-2}v^{j-2}v,_{ii\ell}v,_\ell=48\binom{j+2}{j-2}v^{j-2}v,_{ii\ell}v,_\ell,\nonumber
\\
\mathcal{N}^j_4&=4\binom{j+2}{j}\binom{j}{j-2}v^{j-2}v,_{i\ell}v,_{i\ell}=24\binom{j+2}{j-2}v^{j-2}v,_{i\ell}v,_{i\ell},\nonumber
\\
\mathcal{N}^j_5&=24\binom{j+2}{j}\binom{j}{j-3}v^{j-3}v,_{i\ell}v,_\ell v,_i=240\binom{j+2}{j-3}v^{j-3}v,_{i\ell}v,_\ell v,_i,\nonumber
\\
\mathcal{N}^j_6&=12\binom{j+2}{j}\binom{j}{j-3}v^{j-3}v,_\ell v,_\ell v,_{ii}=120\binom{j+2}{j-3}v^{j-3}v,_\ell v,_\ell v,_{ii},\nonumber
\\
\mathcal{N}^j_7&=24\binom{j+2}{j}\binom{j}{j-4}v^{j-4}v,_\ell v,_\ell v,_i v,_i=360\binom{j+2}{j-4}v^{j-4}v,_\ell v,_\ell v,_i v,_i,\nonumber
\end{align}
thanks to the binomial coefficient identity
\[
\binom{n}{m}\binom{m}{k}=\binom{n}{k}\binom{n-k}{m-k}.
\]
Again, we omit to write the time variable when not necessary. The previous equality \eqref{eq2} in Fourier variables reads as 
\begin{equation}\label{feq2}
\pat\fv(k)=-3|k|^4\fv(k)+\sum_{m=1}^{7}\sum_{j= 2}^N(-1)^j\fa^j_m(k) 
\end{equation}
where 
\begin{align}
\fa^j_1(k)&=3\binom{j+2}{j-1}\sum_{\al^1\in\ZZ}\ldots\sum_{\al^{j-1}\in\ZZ} \fv(\al^{j-1})\prod_{n=1}^{j-2}\fv(\al^n-\al^{n+1})|k-\al^1|^4\fv(k-\al^1),\nonumber
\\
\fa^j_2(k)&=12\binom{j+2}{j-2}\sum_{\al^1\in\ZZ}\ldots\sum_{\al^{j-1}\in\ZZ} \fv(\al^{j-1})\prod_{n=2}^{j-2}\fv(\al^n-\al^{n+1})
|\al^1-\al^2|^2\fv(\al^1-\al^2)
|k-\al^1|^2\fv(k-\al^1),\nonumber
\\
\fa^j_3(k)&=48\binom{j+2}{j-2}\sum_{\al^1\in\ZZ}\ldots\sum_{\al^{j-1}\in\ZZ}\fv(\al^{j-1})\prod_{n=2}^{j-2}\fv(\al^n-\al^{n+1})
|\al^1-\al^2|^2(\al^1_\ell-\al^2_\ell)\nonumber\\
&\qq\qq\qq\qq\qq\qq\qq\qq\qq\times\fv(\al^1-\al^2)
(k_\ell-\al^1_\ell)\fv(k-\al^1),\nonumber
\\
\fa^j_4(k)&=24\binom{j+2}{j-2}\sum_{\al^1\in\ZZ}\ldots\sum_{\al^{j-1}\in\ZZ}\fv(\al^{j-1})\prod_{n=2}^{j-2}\fv(\al^n-\al^{n+1})
(\al^1_\ell-\al^2_\ell)(\al^1_i-\al^2_i)\fv(\al^1-\al^2)\nonumber\\
&\qq\qq\qq\qq\qq\qq\qq\qq\qq\times
(k_\ell-\al^1_\ell)(k_i-\al^1_i)\fv(k-\al^1),\nonumber
\\
\fa^j_5(k)&=240\binom{j+2}{j-3}\sum_{\al^1\in\ZZ}\ldots\sum_{\al^{j-1}\in\ZZ} \fv(\al^{j-1})\prod_{n=3}^{j-2}\fv(\al^n-\al^{n+1})
(\al^2_\ell-\al^3_\ell)(\al^2_i-\al^3_i)\fv(\al^2-\al^3)\nonumber
\\
&\qq\qq\qq\qq\qq\qq\qq\qq\qq\times(\al^1_\ell-\al^2_\ell)\fv(\al^1-\al^2)
(k_i-\al^1_i)\fv(k-\al^1),\nonumber
\\
\fa^j_6(k)&=120\binom{j+2}{j-3}\sum_{\al^1\in\ZZ}\ldots\sum_{\al^{j-1}\in\ZZ} \fv(\al^{j-1})\prod_{n=3}^{j-2}\fv(\al^n-\al^{n+1})
(\al^2_\ell-\al^3_\ell)\fv(\al^2-\al^3)\nonumber\\
&\qq\qq\qq\qq\qq\qq\qq\qq\qq\times
(\al^1_\ell-\al^2_\ell)\fv(\al^1-\al^2)
|k-\al^1|^2\fv(k-\al^1),\nonumber
\\
\fa^j_7(k)&=360\binom{j+2}{j-4}\sum_{\al^1\in\ZZ}\ldots\sum_{\al^{j-1}\in\ZZ} \fv(\al^{j-1})\prod_{n=4}^{j-2}\fv(\al^n-\al^{n+1})
(\al^3_\ell-\al^4_\ell)\fv(\al^3-\al^4)
(\al^2_\ell-\al^3_\ell)\fv(\al^2-\al^3)\nonumber\\
&\qq\qq\qq\qq\qq\qq\qq\qq\qq\times
(\al^1_i-\al^2_i)\fv(\al^1-\al^2)
(k_i-\al^1_i)\fv(k-\al^1).\nonumber
\end{align}
We recall \eqref{dert}, reintroduce the time variable in the notation and estimate \eqref{feq2} as follows:
\begin{equation}\label{feq2dis}
\pat|v(t)|_0\le-3|v(t)|_4+\sum_{m=1}^{7}\sum_{j= 2}^\infty|\mathcal{N}^j_m(t)|_0.
\end{equation}
We are left with the estimates of $|N^j_m(t)|_0$. We first apply Tonelli's Theorem and then the interpolation inequality \eqref{interpolation} to obtain that
\begin{align}
|\mathcal{N}^j_1(t)|_0&\le 3\binom{j+2}{j-1} |v(t)|_0^{j-1}|v(t)|_4=\frac{(j+2)(j+1)j}{2}|v(t)|_0^{j-1}|v(t)|_4,\nonumber
\\
|\mathcal{N}^j_2(t)|_0&\le 12\binom{j+2}{j-2}|v(t)|_0^{j-2}|v(t)|_2^2
\le \frac{(j+2)(j+1)j(j-1)}{2}|v(t)|_0^{j-1}|v(t)|_4,\nonumber
\\
|\mathcal{N}^j_3(t)|_0&
\le 48\binom{j+2}{j-2}|v(t)|_0^{j-2}|v(t)|_3|v(t)|_1
\le 2(j+2)(j+1)j(j-1)|v(t)|_0^{j-1}|v(t)|_4,\nonumber\\
|\mathcal{N}^j_4(t)|_0&\le
 24\binom{j+2}{j-2}|v(t)|_0^{j-2}|v(t)|_2^2
\le (j+2)(j+1)j(j-1)|v(t)|_0^{j-1}|v(t)|_4,\nonumber
\\
|\mathcal{N}^j_5(t)|_0&\le 240\binom{j+2}{j-3}|v|_0^{j-3}|v(t)|_2|v(t)|_1^2
\le 2(j+2)(j+1)j(j-1)(j-2)|v(t)|_0^{j-1}|v(t)|_4,\nonumber
\\
|\mathcal{N}^j_6(t)|_0&\le 120\binom{j+2}{j-3}|v(t)|_0^{j-3}|v(t)|_1^2|v(t)|_2
\le (j+2)(j+1)j(j-1)(j-2)|v(t)|_0^{j-1}|v(t)|_4,\nonumber
\\
|\mathcal{N}^j_7(t)|_0&\le 360\binom{j+2}{j-4}|v(t)|_0^{j-4}|v(t)|_1^4
\le 
\frac{(j+2)(j+1)j(j-1)(j-2)(j-3)}{2}|v(t)|_0^{j-1}|v(t)|_4
.\nonumber
\end{align}
Thus, 
\begin{align}
&\sum_{m=1}^{7}|\mathcal{N}^j_m(t)|_0\nonumber\\
&\le 
\frac{(j+2)(j+1)j}{2}
\left[
1+7(j-1)+6(j-1)(j-2)+(j-1)(j-2)(j-3)
\right]|v(t)|_0^{j-1}|v(t)|_4.\nonumber
\end{align}
Collecting the previous steps, we have
\begin{align}
\pat|v(t)|_0&\le |v(t)|_4\left\{\sum_{j= 2}^\infty\frac{(j+2)(j+1)j}{2}
\left[
1+7(j-1)+6(j-1)(j-2)\right.\right.\nonumber\\
&\qq\qq\qq\qq\qq\qq\q \left.\left.+(j-1)(j-2)(j-3)
\right]|v(t)|_0^{j-1}
-3\right\}.\nonumber
\end{align}
We recall the following identities for power series
\begin{align}
\sum_{j\ge2} (j+2)(j+1)j\,w^{j-1}&=\left(w^2 \left(\frac{1}{1-w}-1-w\right)\right)^{'''}\nonumber
\\
&=\frac{3!}{(1-w)^4}-3!,\nonumber
\\
w\sum_{j\ge2} (j+2)(j+1)j(j-1)\,w^{j-2}&=w\left(w^2 \left(\frac{1}{1-w}-1-w\right)\right)^{iv)}\nonumber
\\
&=4!\frac{w}{(1-w)^5},\nonumber
\\
w^2\sum_{j\ge3} (j+2)(j+1)j(j-1)(j-2)\,w^{j-3}&=w^2\left(w^2 \left(\frac{1}{1-w}-1-w-w^2\right)\right)^{v)}\nonumber
\\
&=5!\frac{w^2}{(1-w)^6},\nonumber
\\
w^3\sum_{j\ge4} (j+2)(j+1)j(j-1)(j-2)(j-3)\,w^{j-4}&=w^3\left(w^2 \left(\frac{1}{1-w}-1-w-w^2-w^3\right)\right)^{vi)}\nonumber
\\
&=6!\frac{w^3}{(1-w)^7}.\nonumber
\end{align}
We set 
\begin{align}
&\de(|v(t)|_0)\nonumber\\
&:=3-\sum_{j\ge 2}\frac{(j+2)(j+1)j}{2}
\left[
1+7(j-1)+6(j-1)(j-2)+(j-1)(j-2)(j-3)
\right]|v(t)|_0^{j-1},\nonumber\\
&=6-\frac{3}{(1-|v(t)|_0)^{4}}\left( 1+28\frac{|v(t)|_0}{(1-|v(t)|_0)}+120\frac{|v(t)|_0^2}{(1-|v(t)|_0)^2}+120\frac{|v(t)|_0^3}{(1-|v(t)|_0)^3}
\right),\nonumber
\end{align}
so that, thanks to \eqref{small2} and reasoning as in Proposition \ref{propo1} (see the $t^*$ argument), we are allowed to integrate in time \eqref{feq2dis} and thanks to the definition of $\de(\cdot)$, we get
\begin{equation*}
|v(t)|_0 +\de(|v_0|_0)\int_0^t|v(s)|_4\,ds\le |v_0|_0
\end{equation*}
and the uniform boundedness in $\mL^1(0,T;\mathscr{{\mA}}^4(\TT))$ follows together with the exponential decay in \eqref{decexp2}.

The uniform boundedness $ \mW^{1,1}(0,T; \mA^0(\TT))$ follows from the inequality
\begin{equation*}
|\pat v(t)|_0\le C|v(t)|_4.
\end{equation*}
\end{proof}

\subsection{Convergence of the approximate problems}

\begin{proposition}[Compactness results]\label{compacidad2}
Let $v_0\in \mA^0(\TT)$ be a function satisfying the condition in \eqref{small2} and $\partial$ be any differential operator of order one. Then, up to subsequences, we  have that every approximating sequence of solutions $\{v_N\}_N$ of \eqref{pbv2reg} verifies
\begin{align} 
\{ v_N\}_N\overset{*}{\rightharpoonup}  v\q &\text{in} \q \mL^{\infty}(0,T;\mL^\infty(\TT)),\label{comp1b}\\
\{ \partial^4 v_N\}_N\overset{*}{\rightharpoonup}  \partial^4 v\q& \text{in} \q \mathcal{M}(0,T;\mL^\infty(\TT)),\label{comp2b}\\
\{\partial^3 v_N\}_N\overset{*}{\rightharpoonup}  \partial^3 v\q& \text{in} \q \mL^{\frac{4}{3}}(0,T;\mL^{\infty}(\TT)),\label{comp3b}\\
\{ v_N\}_N\rightharpoonup v\q &\text{in} \q \mL^2(0,T,\mH^2(\TT)),\label{comp4b}\\
\{ v_N\}_N\rightarrow v\q &\text{in}\q\mL^2(0,T;\mH^r(\TT))\q\text{with}\q0\leq r<2.\label{comp5b}
\end{align}
Furthermore,
$$
v\in \mathscr{C}([0,T];\mL^2(\TT)).
$$
\end{proposition}
\begin{proof}
Being the approximating solutions $\{v_N\}_N$ uniformly bounded in 
\begin{equation*} 
\mL^\infty(0,T; \mA^0(\TT))\cap \mL^1(0,T;\mathscr{\mA}^4(\TT))
\end{equation*}
thanks to Proposition \ref{sap2}, we can reason as in Proposition \ref{propo2} in order to prove that \eqref{comp1b}, \eqref{comp2b}, \eqref{comp3b} and \eqref{comp4b} hold. As far as \eqref{comp5b} is concerned we need  
\[
\{ v_N\}_N\rightarrow v\q \text{ in }\mL^2(0,T;\mL^2(\TT)).
\]
In this way, we can reason as in Proposition \ref{propo2} and prove a uniform bound for $\{ \pat v_N\}_N$ in $\mL^2(0,T, \mH^{-2}(\TT))$ so the assertion follows from \cite[Corollary $4$]{S}.
We compute
\begin{align}
\left|\int_{\TT}\pat v_N(t)\vp\,dx\right|
&=\left|\int_{\TT}\D \left(\sum_{j= 0}^N(-1)^j\binom{j+2}{j}v_N^{j}(t)\right)\,\D\vp\,dx\right|\nonumber\\
&\le \frac{1}{2}
\left[\int_{\TT}|\D v_N(t)||\Delta\vp|
 \sum_{j= 0}^N(j+2)(j+1)j|v_N(t)|^{j-1}\,dx\right.\nonumber\\
&\qq\qq\qq+\left.
\int_{\TT}|\N v_N(t)|^2|\Delta\vp|
 \sum_{j= 0}^N(j+2)(j+1)j(j-1)|v_N(t)|^{j-2}\,dx\label{riesz2}
\right],
\end{align}
with $\vp\in  \mH^2(\TT)$ and such that $\|\vp\|_{ \mH^2}\le 1$. We recall the following expressions (valid for for $0<w<1$) for finite sums
\begin{align}
\sum_{j=0}^{N}(j+2)(j+1)j\,w^{j-1}&=
\left(w^2 \frac{1-w^{N+1}}{1-w} \right)^{'''}=w^2\left( \frac{1-w^{N+1}}{1-w} \right)^{'''}+6w\left( \frac{1-w^{N+1}}{1-w} \right)^{''}\nonumber\\
&\q+
6\left( \frac{1-w^{N+1}}{1-w} \right)^{'},\nonumber\\
\sum_{j=0}^{N}(j+2)(j+1)j(j-1)\,w^{j-2}&=\left(w^2 \frac{1-w^{N+1}}{1-w} \right)^{iv)}=w^2\left( \frac{1-w^{N+1}}{1-w} \right)^{iv)}+8w\left( \frac{1-w^{N+1}}{1-w} \right)^{'''}\nonumber\\
&\q+
12\left( \frac{1-w^{N+1}}{1-w} \right)^{''}.\nonumber
\end{align}
For $0<w<1$, it holds that
\begin{align}
\left( \frac{1-w^{N+1}}{1-w} \right)^{'}&\le
\frac{1-w^{N+1}}{(1-w)^2},\nonumber
\\
\left( \frac{1-w^{N+1}}{1-w} \right)^{''}&\le 
2\,\frac{1-w^{N+1}}{(1-w)^3},\nonumber
\\
\left( \frac{1-w^{N+1}}{1-w} \right)^{'''}&\le
6\,\frac{1-w^{N+1}}{(1-w)^{4}},\nonumber
\\
\left( \frac{1-w^{N+1}}{1-w} \right)^{iv)}&\le 
24\,\frac{1-w^{N+1}}{(1-w)^{5}}.\nonumber
\end{align}
We thus estimate
\begin{align}
\sum_{j=0}^{N}(j+2)(j+1)j\,w^{j-1}&\le
6w^2\,\frac{1-w^{N+1}}{(1-w)^{4}}
+12w\,\frac{1-w^{N+1}}{(1-w)^3}+
6\,\frac{1-w^{N+1}}{(1-w)^2}\nonumber\\
&=6\,\frac{1-w^{N+1}}{(1-w)^{4}}\nonumber\\
&\le \frac{6}{(1-w)^{4}},\nonumber\\
\sum_{j=0}^{N}(j+2)(j+1)j(j-1)\,w^{j-2}&\le
24w^2\,\frac{1-w^{N+1}}{(1-w)^{5}}
+48w\,\frac{1-w^{N+1}}{(1-w)^{4}}+
 24\,\frac{1-w^{N+1}}{(1-w)^3}\nonumber\\
&=24\,\frac{1-w^{N+1}}{(1-w)^{5}}\nonumber\\
&\le \frac{24}{(1-w)^{5}}.
\end{align}
We come back to \eqref{riesz2} and, letting $w=|v_N(t)|$, we deduce
\begin{align}
\left|\int_{\TT}\pat v_N(t)\vp\,dx\right|&\le
24 
\int_{\TT}\left[|\D v_N(t)| + |\N v_N(t)|^2\right]\frac{|\Delta\vp|}{(1-|v_N(t)|)^{5}}\, dx\nonumber\\
&\le  \frac{24\,\|\vp\|_{ \mH^2}}{(1-|v_N(t)|_0)^{5}}(\|v_N(t)\|_{ \mH^2}+\|\nabla v_N(t)\|_{\mL^4}^2)\nonumber\\
&\le\frac{24\|\vp\|_{ \mH^2}}{(1-|v_N(t)|_0)^{4}}\|v_N(t)\|_{ \mH^2},\nonumber
\end{align}
thanks to \eqref{smart1} too.\\
Finally, we get
\[
\|\pat v_N(t)\|_{ \mH^{-2}}^2\leq\sup_{
\footnotesize
\begin{array}{c}
\vp\in  \mH^2(\TT)\\ \|\vp\|_{ \mH^2}\le 1
\end{array}
}\left|\int_{\TT}\pat v_N(t)\vp\,dx\right|^2\leq 
\frac{24^2}{(1-|v_N(t)|_0)^{8}}\|v_N(t)\|_{ \mH^2}^2
\]
from which we deduce that
\[
\int_0^T \|\pat v_N(t)\|_{ \mH^{-2}}^2\,dt\le \frac{24^2}{(1-\|v_N\|_{{\mL^\infty(\mA^0)}})^{8}}\|v_N\|_{\mL^2( \mH^2)}^2<c
\]
and 
\[
\{\pat v_N\}_N\q\text{is uniformly bounded in}\q   \mL^2(0,T;\mH^{-2}(\TT)).
\]

\end{proof}

\subsection{Proof of Theorem \ref{teo2ex}}\label{proofteo2ex}
Let $\vp\in C^\infty([0,T)\times\TT)$. Again, we observe that
\[
\int_0^T\int_\TT\left( \sum_{j= 0}^N(-1)^j\binom{j+2}{j} v_N^j
-\frac{1}{(1+v_N)^3}\right)\D^2\vp\,dx\,dt\to 0
\]
when $N\to\infty$. Then, we consider 
\begin{align*}
&\int_\TT \left| \frac{1}{(1+v_N(t))^3}-\frac{1}{(1+v(t))^3} \right||\D^2\vp(t)|\,dx\\
&\le
\int_\TT
\frac{|v^3(t)-v_N^3(t)|+3|v^2(t)-v_N^2(t)|+3|v(t)-v_N(t)|}{|1+v_N(t)|^3|1+v(t)|^3}
|\D^2\vp(t)|\,dx\\
&\le c(\|v_N(t)\|_{\mL^\infty},\|v(t)\|_{\mL^\infty})\|v(t)-v_N(t)\|_{\mL^2}\|\D^2\vp(t)\|_{\mL^2}
\end{align*}
and the claim follows thanks to the convergence results proved in Proposition \ref{compacidad2}.\\
The regularity in \eqref{ex2} follows from Proposition \ref{compacidad2} and the exponential decay in \eqref{decA0} can be deduced from \eqref{decexp2} reasoning as in Subsection \ref{pttl1}.

\subsection{Estimates for the regularized problem in the Sobolev space $\mathscr{H}^2$}

\begin{proposition}[Estimates in Sobolev spaces]\label{propo4pb2}
Let $v_0\in\mA^0(\TT)\cap \mH^2(\TT)$ be a function satisfying the condition \eqref{small2}. Then, every approximating sequence of solutions $\{v_N\}_N$ of \eqref{2v} is uniformly bounded in
\[
\mL^\infty(0,T; \mH^2(\TT))\cap \mL^2(0,T;\mathscr{\mH}^4(\TT)).
\]
Furthermore, we have that
\[
\|v_N(t)\|_{\mH^2}\le C(|v_0|_0)\qq\forall t\in[0,T].
\]
\end{proposition}

\begin{proof}[Sketch of the proof]
The proof goes as follows: first we multiply the equation by $\Delta^2 v$ and integrate by parts. As in the proof of Proposition \ref{propo4}, we have to integrate by parts (several times) to find a perfect derivative hidden in the \emph{medium order} terms. Once this structure of perfect derivative is obtained, we can integrate by parts once again and find terms akin to
$$
\int\partial(v^n\partial v)|\partial^3 v|^2.
$$
These terms can be handled as in Proposition \ref{propo4}.
\end{proof}

\begin{proof}[Proof of Proposition \ref{propo4pb2}]

Again, we write $v$ instead of $v_N$ and omit the time variable if not needed. We multiply the main equation in \eqref{2v} by $\D^2 v$ obtaining
\begin{align}
\frac{1}{2}\frac{d}{dt}\int_{\TT}|\D v|^2\,dx&=\int_{\TT}\D^2\left(
\sum_{n=0}^N\binom{n+2}{n}(-v)^n
\right)\D^2 v\,dx\nonumber\\
&=
-3\int_\TT \sum_{n=1}^N\binom{n+2}{n-1}(-v)^{n-1}|v,_{jjrr}|^2\,dx
\nonumber\\
&\q+12\int_{\TT}\sum_{n=2}^N\binom{n+2}{n-2}(-v)^{n-2}\left[
4v,_{ii\ell}v,_{\ell}+|v,_{ii}|^2+2v,_{i\ell}v,_{i\ell}
\right]\,dx
\nonumber\\
&\q-60\int_{\TT}\sum_{n=2}^N\binom{n+2}{n-3}(-v)^{n-3}\left[
v,_{ii}v,_\ell v,_\ell+2v,_{i\ell}v,_\ell v,_i
\right]\,dx\nonumber\\
&\q+360\int_{\TT}\sum_{n=2}^N\binom{n+2}{n-4}(-v)^{n-4}v,_iv,_i v,_\ell v,_\ell
v,_{jjrr}\,dx\nonumber\\
&=-3\int_\TT \sum_{n=1}^N\binom{n+2}{n-1}(-v)^{n-1}v,_{jjrr}\,dx+\sum_{m=1}^{6}I_m.\label{eqreg}
\end{align}
Again, we have that 
\[
-\int_\TT \sum_{n=1}^N\binom{n+2}{n-1}(-v)^{n-1}|v,_{jjrr}|^2\,dx<0
\]
for $N$ large since
\[
\sum_{n=1}^N\binom{n+2}{n-1}(-v)^{n-1}-\frac{1}{(1+v)^4}\underset{N\to\infty}{\longrightarrow} 0.
\]
Recalling that $|v_0|_0<1$ and using also \eqref{smart1}, we estimate
\begin{align}
I_2+I_3&=12\int_{\TT}\sum_{n=2}^N\binom{n+2}{n-2}(-v)^{n-2}
\left(|v,_{ii}|^2+2v,_{i\ell}v,_{i\ell}\right)v,_{jjrr}\,dx\nonumber\\
&\le c(1-|v|_0)^{-5} \|v\|_{ \mH^2}^2|v|_4,\nonumber\\
I_4+I_5&=-60\int_{\TT}\sum_{n=3}^N\binom{n+2}{n-3}(-v)^{n-3}\left(v,_{ii}v,_\ell v,_\ell+2v,_{i\ell}v,_\ell v,_i \right)v,_{jjrr}\,dx\nonumber\\
&\le c(1-|v|_0)^{-6}\|v\|_{ \mH^2}^2|v|_4,\nonumber
\\
I_6&=360\int_{\TT}\sum_{n=4}^N\binom{n+2}{n-4}(-v)^{n-4}v,_iv,_i v,_\ell v,_\ell v,_{jjrr}\,dx\nonumber\\
&\le c(1-|v|_0)^{-7}\|v\|_{ \mH^2}^2|v|_4.\nonumber
\end{align}
Here, we have also used  the computations contained in Proposition \ref{sap2}. 
We are left with $I_1$:
\[
I_1=48\int_{\TT}\sum_{n=2}^N\binom{n+2}{n-2}(-v)^{n-2}v,_{ii\ell}v,_{\ell}v,_{jjrr}\,dx.
\]
In this later term we have to find a perfect $\ell-$derivative structure that allow us to close the estimates. To do that we have to integrate by parts three times (in $i$, $r$ and $j$). Integrating by parts in $r$, we compute
\begin{align}
\frac{I_1}{48}&= -\int_{\TT} \left(\left(\sum_{n=2}^N\binom{n+2}{n-2}(-v)^{n-2}\right)v,_{ii\ell}v,_{\ell}\right),_rv,_{jjr}\,dx\nonumber\\
&=\int_{\TT} v,_{jjr}\left[ \left(5\sum_{n=3}^N\binom{n+2}{n-3}(-v)^{n-3}\right)v,_{ii\ell}v,_{\ell}v,_r-\left(\sum_{n=2}^N\binom{n+2}{n-2}(-v)^{n-2}\right)v,_{ii\ell}v,_{\ell r}  \right]\,dx\nonumber\\
&\q-\int_{\TT} \left(\sum_{n=2}^N\binom{n+2}{n-2}(-v)^{n-2}\right)v,_{ii\ell r}v,_{\ell}v,_{jjr}\,dx.\nonumber
\end{align}
Now we integrate by parts in $i$:
\begin{align}
\frac{I_1}{48}&=\int_{\TT} v,_{jjr}\left[ \left(5\sum_{n=3}^N\binom{n+2}{n-3}(-v)^{n-3}\right)v,_{ii\ell}v,_{\ell}v,_r -\left(\sum_{n=2}^N\binom{n+2}{n-2}(-v)^{n-2}\right)v,_{ii\ell}v,_{\ell r}  \right]\,dx\nonumber\\
&\q+\int_{\TT} \left(\left(\sum_{n=2}^N\binom{n+2}{n-2}(-v)^{n-2}\right)v,_{\ell}v,_{jjr}\right),_iv,_{i\ell r}\,dx\nonumber\\
&=\int_{\TT} v,_{jjr}\left[ \left(5\sum_{n=3}^N\binom{n+2}{n-3}(-v)^{n-3}\right)v,_{ii\ell}v,_{\ell}v,_r -\left(\sum_{n=2}^N\binom{n+2}{n-2}(-v)^{n-2}\right)v,_{ii\ell}v,_{\ell r}  \right]\,dx\nonumber\\
&\q+\int_{\TT} \left(\left(\sum_{n=2}^N\binom{n+2}{n-2}(-v)^{n-2}\right)v,_{\ell}\right),_iv,_{jjr}v,_{i\ell r}\,dx \nonumber\\
&\q +\int_{\TT} \left(\sum_{n=2}^N\binom{n+2}{n-2}(-v)^{n-2}\right)v,_{\ell}v,_{jjri}v,_{i\ell r}\,dx.\nonumber
\end{align}
Integrating by parts in $j$, we find that
\begin{align}
\frac{I_1}{48}&=\int_{\TT} v,_{jjr}\left[ \left(5\sum_{n=3}^N\binom{n+2}{n-3}(-v)^{n-3}\right)v,_{ii\ell}v,_{\ell}v,_r-\left(\sum_{n=2}^N\binom{n+2}{n-2}(-v)^{n-2}\right)v,_{ii\ell}v,_{\ell r}  \right]\,dx\nonumber\\
&\q+\int_{\TT} \left(\left(\sum_{n=2}^N\binom{n+2}{n-2}(-v)^{n-2}\right)v,_{\ell}\right),_iv,_{jjr}v,_{i\ell r}\,dx \nonumber\\
&\q -\int_{\TT} \left(\left(\sum_{n=2}^N\binom{n+2}{n-2}(-v)^{n-2}\right) v,_{i\ell r}v,_{\ell}\right),_jv,_{jri}\,dx\nonumber\\
&=\int_{\TT} v,_{jjr}\left[ \left(5\sum_{n=3}^N\binom{n+2}{n-3}(-v)^{n-3}\right)v,_{ii\ell}v,_{\ell}v,_r -\left(\sum_{n=2}^N\binom{n+2}{n-2}(-v)^{n-2}\right)v,_{ii\ell}v,_{\ell r}  \right]\,dx\nonumber\\
&\q+\int_{\TT} \left(\left(\sum_{n=2}^N\binom{n+2}{n-2}(-v)^{n-2}\right)v,_{\ell}\right),_iv,_{jjr}v,_{i\ell r}\,dx\nonumber\\
&\q  -\int_{\TT} \left(\left(\sum_{n=2}^N\binom{n+2}{n-2}(-v)^{n-2}\right)v,_{\ell}\right),_j v,_{i\ell r}v,_{jri}\,dx\nonumber\\
&\q-\int_{\TT} \left(\sum_{n=2}^N\binom{n+2}{n-2}(-v)^{n-2}\right)v,_{\ell} v,_{i\ell rj}v,_{jri}\,dx\nonumber\\
&=\int_{\TT} v,_{jjr}\left[ \left(5\sum_{n=3}^N\binom{n+2}{n-3}(-v)^{n-3}\right)v,_{ii\ell}v,_{\ell}v,_r -\left(\sum_{n=2}^N\binom{n+2}{n-2}(-v)^{n-2}\right)v,_{ii\ell}v,_{\ell r}  \right]\,dx\nonumber\\
&\q+\int_{\TT} \left(\left(\sum_{n=2}^N\binom{n+2}{n-2}(-v)^{n-2}\right)v,_{\ell}\right),_iv,_{jjr}v,_{i\ell r}\,dx\nonumber\\
&\q  -\int_{\TT} \left(\left(\sum_{n=2}^N\binom{n+2}{n-2}(-v)^{n-2}\right)v,_{\ell}\right),_j v,_{i\ell r}v,_{jri}\,dx\nonumber\\
&\q+\frac{1}{2}\int_{\TT} \left(\left(\sum_{n=2}^N\binom{n+2}{n-2}(-v)^{n-2}\right)v,_{\ell}\right),_\ell v,_{i rj}v,_{jri}\,dx.\nonumber
\end{align}
Thus,
\begin{align}
I_1&=48\int_{\TT} v,_{jjr}\left[ \left(5\sum_{n=3}^N\binom{n+2}{n-3}(-v)^{n-3}\right)v,_{ii\ell}v,_{\ell}v,_r-\left(\sum_{n=2}^N\binom{n+2}{n-2}(-v)^{n-2}\right)v,_{ii\ell}v,_{\ell r}  \right]\,dx\nonumber\\
&\q+48\int_{\TT} \left(\left(\sum_{n=2}^N\binom{n+2}{n-2}(-v)^{n-2}\right)v,_{\ell}\right),_iv,_{jjr}v,_{i\ell r}\,dx\nonumber\\
&\q-48\int_{\TT} \left(\left(\sum_{n=2}^N\binom{n+2}{n-2}(-v)^{n-2}\right)v,_{\ell}\right),_j v,_{i\ell r}v,_{jri}\,dx\nonumber\\
&\q+24\int_{\TT} \left(\left(\sum_{n=2}^N\binom{n+2}{n-2}(-v)^{n-2}\right)v,_{\ell}\right),_\ell v,_{i rj}v,_{jri}\,dx.\nonumber
\end{align}
We just deal with 
\[
\int_{\TT}  v,_{jjr}  \left(\sum_{n=3}^N\binom{n+2}{n-3}(-v)^{n-3}\right)v,_{ii\ell}v,_{\ell}v,_r\,dx,
\]
since the other terms can be studied in the same way. We have that
\begin{align}
\int_{\TT} v,_{jjr}  \left(\sum_{n=3}^N\binom{n+2}{n-3}(-v)^{n-3}\right)v,_{ii\ell}v,_{\ell}v,_r\,dx&\le (1-|v(t)|_0)^{-6}
\int_{\TT} |v,_{jjr} || v,_{ii\ell}||v,_{\ell}||v,_r|\,dx\nonumber\\
&\le (1-|v|_0)^{-6}\|\N v\|_{\mL^4}^2 \|\N \D v\|_{\mL^4}^2.\nonumber
\end{align}
Then, applying \eqref{smart1} on $\|\N v(t)\|_{\mL^4}^2$ and \eqref{smart2} on $\|\N \D v(t)\|_{\mL^4}^2$, we get
\begin{equation*}
\int_{\TT} v,_{jjr}  \left(\sum_{n=3}^N\binom{n+2}{n-3}(-v)^{n-3}\right)v,_{ii\ell}v,_{\ell}v,_r\,dx\le c (1-|v|_0)^{-6}
 \|  v\|_{ \mH^2}^2|v|_4.
\end{equation*}
Finally
\[
I_1\le c (1-|v|_0)^{-6}
 \|  v\|_{ \mH^2}^2|v|_4.
\]

We come back to the notation $v_N$, reintroduce the $t$ in the notation and gather the previous estimates in the following inequality:
\[
\frac{d}{dt}\|v_N(t)\|_{\mH^2}^2\le 
c (1-|v_N(t)|_0)^{-7}
\left(
(1-|v_N(t)|_0)^2+(1-|v_N(t)|_0)+1
\right)
 \|  v_N(t)\|_{ \mH^2}^2|v_N(t)|_4
\]
from which
\begin{equation*} 
\esssup_{0\leq t\leq T}\|v_N(t)\|_{ \mH^2}^2\le e^{c(|v_0|_0)\int_0^t|v_N(s)|_4\,ds}\|v_0\|_{ \mH^2}^2\le c(|v_0|_0),
\end{equation*}
thanks to Gronwall's inequality.

The uniform boundedness in $\mL^2(0,T;\mH^4(\TT))$ follows from  \eqref{eqreg} and from the inequality
\[
\sum_{n=1}^N\binom{n+2}{n-1}(-v_N)^{n-1}-(1+v)^{-4}+(1+v)^{-4}>(1+v)^{-4}-\eps>0,
\]
which holds for $N$ large and $0<\eps\ll1$, since
\begin{align}
&3\left[(1+\|v_N\|_{\mL^\infty(\mL^\infty)})^{-4}-\eps\right]\int_0^T\int_\TT |(v_N),_{jjrr}|^2\,dx\,dt \nonumber\\
&\le 
3\int_0^T\int_\TT \sum_{n=1}^N\binom{n+2}{n-1}(-v_N)^{n-1}|(v_N),_{jjrr}|^2\,dx\,dt\nonumber\\
&\le \int_0^T \sum_{m=1}^{6}I_m(t)\,dt<c.\nonumber
\end{align}
\end{proof}

\subsection{Convergence of the approximate problems}
\begin{proposition}[Compactness results]\label{compacidad3}
Let $v_0\in \mA^0(\TT)\cap \mH^2(\TT)$ be a function satisfying the condition in \eqref{small2} and $\partial$ be any differential operator of order one. Then, up to subsequences, we  have that every approximating sequence of solutions $\{v_N\}_N$ of \eqref{pbv2reg} verifies
\begin{equation}\label{L2Hr2}
\{ v_N\}_N\rightarrow v\q \text{ in }\mL^2(0,T;\mH^r(\TT)),\;0\leq r<4.
\end{equation}
Furthermore, the limit function $v$ satisfies
$$
v\in\mathscr{C}([0,T];\mH^2(\TT)).
$$
\end{proposition}
\begin{proof}
The proof of \eqref{L2Hr2} reasoning as in Proposition \ref{propo5} (see also Proposition \ref{propo2}). More precisely, an interpolation in Sobolev spaces and the uniform boundedness in $\mL^2(0,T;\mH^4(\TT))$ lead to
\begin{align}
\int_{0}^T\|v_N(t)-v(t)\|_{\mH^r}^2\,dt&\leq \int_{0}^T\|v_N(t)-v(t)\|_{\mL^2}^{2-r}\|v_N(t)-v(t)\|_{\mH^4}^{r}\,dt\nonumber\\
&\leq
\left(
\int_{0}^T\|v_N(t)-v(t)\|_{\mL^2}^2\,dt
\right)^\frac{2-r}{2}
\left(
\int_{0}^T\|v_N(t)-v(t)\|_{\mH^4}^2\,dt
\right)^\frac{r}{2}\nonumber
\end{align}
for every $0\leq r<4$. Thus \eqref{L2Hr2} follows thanks also to the strong convergence in $\mL^2(0,T;\mL^2(\TT))$. The continuity follows from 
\[
\D v\in \mL^2(0,T; \mH^2(\TT)),\qq \pat\D v\in \mL^2(0,T; \mH^{-2}(\TT)).
\]
\end{proof}

\subsection{Proof of Theorem \ref{teo2reg}}
The proof of \eqref{reg1b} follows taking advantage of Propositions \ref{propo4pb2} and \ref{compacidad3}. The inequality in \eqref{H2H4b} can be deduced from Proposition \ref{propo4pb2}. Indeed it provides us with 
\[
\frac{1}{2}\int_{\TT}|\D v_N(t)|^2\,dx+3(1-\sup_{t\in[0,T]}|v_N(t)|_0)^{-4}\int_0^T\int_\TT |\D^2v_N|^2\,dx\,dt\le c
\]
and passing to the limit as in Subsection \ref{proofteo2ex}. Finally, the exponential decay \eqref{expHrb} follows as in Subsection \ref{proofreg1}.


\section*{Acknowledgements}The first author was partially supported by the LABEX MILYON (ANR-10-LABX-0070) of Universit\'e de Lyon, within the program "Investissements d'Avenir" (ANR-11-IDEX-0007) operated by the French National Research Agency (ANR). Part of the research leading to results presented here was conducted during a short stay of the second author at Institut Camille Jordan.


\end{document}